\documentclass[12pt]{amsart}
\usepackage{amssymb,bm}
\usepackage[hidelinks]{hyperref}
\usepackage{etoolbox}

\newcommand{\tpdf}{\texorpdfstring}

\newcommand{\Hamm}{\operatorname{Hamm}}
\AtBeginEnvironment{pmatrix}{\setlength{\arraycolsep}{50pt}}
\newcommand{\colvec}[1]{\begingroup
	\renewcommand*{\arraystretch}{0.8} \begin{pmatrix} #1 \end{pmatrix}
	\endgroup}

\newcommand{\mb}{\mathbf}

\newtheorem{theorem}{Theorem}
\newtheorem{lemma}[theorem]{Lemma}
\newtheorem{proposition}[theorem]{Proposition}
\newtheorem{corollary}[theorem]{Corollary}
\newtheorem*{claim}{Claim}

\newtheorem{conjecture}[theorem]{Conjecture}
\numberwithin{theorem}{section}
\numberwithin{equation}{section}

\theoremstyle{definition}

\newtheorem{observation}[theorem]{Observation}
\newtheorem{definition}[theorem]{Definition}
\newtheorem{notation}[theorem]{Notation}

\newtheorem{remark}[theorem]{Remark}

\title[A set of 2-recurrence]{A set of 2-recurrence whose perfect squares do not form a set of measurable recurrence}

\author{John T. Griesmer}
\email{jtgriesmer@gmail.com}
\address{Department of Applied Mathematics and Statistics, Colorado School of Mines, Golden, Colorado}
\usepackage{amstext}

\begin{document}
	
\begin{abstract}

We say that $S\subseteq\mathbb Z$ is a set of $k$-recurrence if for every measure preserving transformation $T$ of a probability measure space $(X,\mu)$ and every $A\subseteq X$ with $\mu(A)>0$, there is an $n\in S$ such that $\mu(A\cap T^{-n} A\cap T^{-2n}\cap \dots \cap T^{-kn}A)>0$.  A set of $1$-recurrence is called a set of measurable recurrence.

Answering a question of Frantzikinakis, Lesigne, and Wierdl, we construct a set of $2$-recurrence $S$ with the property that $\{n^2:n\in S\}$ is not a set of measurable recurrence.
\end{abstract}

\maketitle

\section{Background and motivation}

A \emph{probability measure preserving system} (or \emph{MPS}) is a quadruple $(X,\mathcal B,\mu,T)$ where $(X,\mathcal B,\mu)$ is a probability measure space and $T:X\to X$ is an invertible transformation preserving $\mu$, meaning $\mu(T^{-1}A)=\mu(A)$ for every measurable set $A\subseteq X$.

We say that $S\subseteq \mathbb Z$ is a \emph{set of measurable recurrence} if for every MPS $(X,\mathcal B,\mu,T)$ and every $A\subseteq X$ having $\mu(A)>0$, there is an $n\in S$ such that $\mu(A\cap T^{-n}A)>0$.

For a fixed $k\in \mathbb N$, we say $S$ is a \emph{set of $k$-recurrence} if under these hypotheses, there is an $n\in S$ such that $\mu\bigl(\bigcap_{j=0}^{k} T^{-jn}A\bigr)>0$; in this terminology a set of measurable recurrence is a set of $1$-recurrence.

Finally, $S\subseteq \mathbb Z$ is a \emph{set of Bohr recurrence} if for all $d\in \mathbb N$, every  $\bm\alpha \in \mathbb T^d$, and all $\varepsilon>0$, there is an $n\in S$ such that $\|n\bm\alpha\|<\varepsilon$ (see \S\ref{sec:Definitions} for definitions and notation).  

Frantzikinakis, Lesigne, and Wierdl \cite{FLW2006} proved that if $k\in \mathbb N$ and $S\subseteq \mathbb Z$ is a set of $k$-recurrence, then $S^{\wedge k}:=\{n^k: n\in S\}$ is a set of Bohr recurrence.   They ask\footnote{Remarks following Proposition 2.2 of \cite{FLW2006}.} whether this conclusion can be strengthened to ``$S^{\wedge k}$ is a set of measurable recurrence,'' and the subsequent articles \cite{FLW2009} and \cite{FrProblems} reiterate\footnote{Problem 5 of the current version of \cite{FrProblems} at \href{https://arxiv.org/abs/1103.3808}{arXiv:1103.3808}.} this question.  Our main result, Theorem \ref{thm:Main}, provides a negative answer for the case $k=2$.  For $k\geq 3$, the question remains open. A related question in \cite{FLW2009} asks whether a set $S$ which is a set of $k$-recurrence for every $k$ must have the property that $S^{\wedge 2}$ is a set of measurable recurrence. We discuss how our construction relates to these questions in \S\ref{sec:Remarks}.

\begin{theorem}\label{thm:Main}
There is a set $S\subseteq \mathbb Z$ which is a set of $2$-recurrence such that $S^{\wedge 2}$ is not a set of measurable recurrence.
\end{theorem}

Reflecting on the known examples of sets of Bohr recurrence which are not sets of measurable recurrence, Frantzikinakis \cite{FrProblems} predicts that an example of a set of $2$-recurrence $S$ where $S^{\wedge 2}$ is not a set of measurable recurrence will be rather complicated.  Our example is indeed complicated: while built from well known constituents  using standard methods, the proof that it is a set of $2$-recurrence uses several reductions - from general measure preserving systems to totally ergodic systems to nilsystems to affine systems to Kronecker systems.  The final reduction combines explicit computations of multiple ergodic averages in 2-step affine systems with classical estimates for three term arithmetic progressions in terms of Fourier coefficients.

\subsection{Outline of the article}  Our approach is similar to Kriz's construction \cite{Kriz} proving that there is a set of topological recurrence which is not a set of measurable recurrence.  Very roughly, our example $S$ in Theorem \ref{thm:Main} is $\{n:n^2 \in R\}$, where $R$ is Kriz's example.  While this description is not quite correct, it may help those familiar with \cite{Kriz}, \cite{GriesmerKrizInDiff}, or \cite{griesmer2020separating} understand our construction.

The overall proof of Theorem \ref{thm:Main} is presented at the end of \S\ref{sec:InductiveStep}.  We outline its components here.  Section \ref{sec:InductiveStep} begins by collecting standard facts about the following finite approximations to recurrence properties.

\begin{definition}\label{def:Finite}
Let $S\subseteq \mathbb Z$ and $k\in \mathbb N$.
We say that $S$ is \emph{$(\delta,k)$-recurrent} if for every MPS $(X,\mathcal B,\mu,T)$  and every $A\subseteq X$ with $\mu(A)>\delta$, we have $A\cap T^{-n}A\cap \dots \cap T^{-kn}A\neq \varnothing$ for some $n\in S$.

We say that $S$ is \emph{$(\delta,k)$-nonrecurrent} if there is an MPS $(X,\mathcal B,\mu,T)$ and $A\subseteq X$ with $\mu(A)>\delta$ such that $A\cap T^{-n}A\cap \dots \cap T^{-kn}A=\varnothing$.

  We say $S$ is \emph{$\delta$-nonrecurrent} if it is $(\delta,1)$-nonrecurrent, meaning there is an MPS $(X,\mathcal B,\mu,T)$ and $A\subseteq X$ with $\mu(A)>\delta$ such that $A\cap T^{-n}A=\varnothing$ for all $n\in S$.
\end{definition}
\begin{remark}
  The condition $A\cap T^{-n}A\cap \dots \cap T^{-kn}A\neq \varnothing$ in the definition of $(\delta,k)$-recurrent may be replaced with $\mu(A\cap T^{-n}A\cap \dots \cap T^{-kn}A)>0$; cf.~Lemma \ref{lem:RecurrenceEquivalence}.
\end{remark}

Lemma \ref{lem:Combine} says that if $S_1, S_2\subseteq \mathbb Z$ are finite, $\delta_1$-nonrecurrent, and $\delta_2$-nonrecurrent, respectively, then for all sufficiently large $m$, $S_1\cup mS_2$ is  $2\delta_1\delta_2$-nonrecurrent.  Thus, if $S_1^{\wedge 2}$ and $S_2^{\wedge 2}$ are $\delta_1$-nonrecurrent and $\delta_2$-nonrecurrent, respectively, then $(S_1 \cup mS_2)^{\wedge 2}$ is $2\delta_1\delta_2$-nonrecurrent for all sufficiently large $m$, as $(S_1 \cup mS_2)^{\wedge 2} = S_1^{\wedge 2} \cup m^2S_2^{\wedge 2}$.

Lemma \ref{lem:RecurrenceCompactness} says that  $S\subseteq \mathbb Z$ is $\delta$-nonrecurrent if and only if for all $\delta'< \delta$ and all \emph{finite} subsets $S'\subseteq S$, $S'$ is $\delta'$-nonrecurrent.  Likewise, if $S\subseteq \mathbb Z$ is $(\eta,2)$-recurrent, then for all $\eta'<\eta$, there is a finite subset $S'\subseteq S$ which is $(\eta',2)$-recurrent.

The proof of Theorem \ref{thm:Main} is given at the end of \S\ref{sec:InductiveStep}; it explains in detail how finite approximations are assembled to form a $2$-recurrent set whose perfect squares do not form a set of measurable recurrence.  This reduces the problem to proving Lemma \ref{lem:FinitePiece}, which states that the required finite approximations exist. These approximations are based on \emph{Bohr-Hamming balls}, which we introduce in Section \ref{sec:Definitions}. Bohr-Hamming balls were used in \cite{Kriz} and \cite{griesmer2020separating} to construct sets with prescribed recurrence properties.  Fixing  $\delta<\frac{1}{2}$ and $\eta>0$,  Lemmas \ref{lem:BHNonrecurrent} and \ref{lem:SqrtBHisrecurrent} show that there is a Bohr-Hamming ball $BH$ which is $\delta$-nonrecurrent, while $\sqrt{BH}:=\{n\in \mathbb N: n^2\in BH\}$ is $(\eta,2)$-recurrent.

The proof of Lemma \ref{lem:SqrtBHisrecurrent} occupies \S\ref{sec:MultipleErgodic}-\S\ref{sec:Appendix}.  It is proved by estimating multiple ergodic averages of the form
\begin{equation}\label{eqn:AverageDef}
\lim_{N\to\infty} \frac{1}{N}\sum_{n=1}^N g(n^2\bm\beta) \int f\cdot f\circ T^n \cdot f\circ T^{2n}\, d\mu,
\end{equation}
where $(X,\mathcal B,\mu,T)$ is a measure preserving system, $f:X\to [0,1]$ has $\int f\, d\mu>\delta$ for some prescribed $\delta>0$,  $\bm\beta\in \mathbb T^r$ for some $r\in \mathbb N$, and $g:\mathbb T^r\to [0,1]$ is Riemann integrable. Under certain hypotheses on $g$, we will prove the limit in (\ref{eqn:AverageDef}) is positive; this is inequality (\ref{eqn:PositiveIntegral}) in the proof of Lemma \ref{lem:SqrtBHisrecurrent}.  In \S\ref{sec:MultipleErgodic} we show how the general case may be reduced to the one where $T$ is totally ergodic. The remainder of the article, outlined in \S\ref{sec:OutlineSpecial}, is dedicated to analyzing the limit in (\ref{eqn:AverageDef}) when $T$ is totally ergodic.  Section \ref{sec:AffineReduction} shows that the totally ergodic case can be further reduced to the study of \emph{standard $2$-step Weyl systems}, and \S\S\ref{sec:Joinings}-\ref{sec:AffineLimits} are dedicated to simplifying and estimating (\ref{eqn:AverageDef}) for these systems.

 Readers familiar with the theory of characteristic factors (especially \cite{FrantzikinakisThreePoly}) may find it most profitable to read \S\ref{sec:InductiveStep}, \S\ref{sec:Definitions}, \S\ref{sec:Outline}, and \S\ref{sec:AffineReduction} in detail, and skim \S\ref{sec:MultipleErgodic}.

\subsection{Acknowledgement}  We thank Nikos Frantzikinakis for helpful comments.  An anonymous referee contributed several corrections and improvements to exposition.

\section{Constructing the example from finite approximations}\label{sec:InductiveStep}

We first require some standard facts about the properties mentioned in Definition \ref{def:Finite}. The following is Lemma 3.6 of \cite{GriesmerKrizInDiff}; it is essentially Lemma 3.2 of \cite{Kriz}. Similar lemmas appear, often unnamed, in the variations on Kriz's example  (\cite{ForrestThesis}, \cite{McCutcheonAlexandria}, \cite{McCutcheonBook}, \cite{WeissBook}).

\begin{lemma}\label{lem:Combine}
	Let $S_1, S_2\subseteq \mathbb N$ be finite. If $S_1$ and $S_2$ are $\delta$-nonrecurrent and $\eta$-nonrecurrent, respectively, then for all sufficiently large $m\in \mathbb N$, $S_1\cup mS_2$ is $2\delta\eta$-nonrecurrent.
\end{lemma}

\begin{lemma}\label{lem:DilatesPreserveRecurrence}
  Let $m\in \mathbb Z$ and $\delta\geq 0$.  If $S\subseteq \mathbb Z$ is $(\delta,2)$-recurrent, then $mS$ is also $(\delta,2)$-recurrent.
\end{lemma}

\begin{proof}
  Fix $m\in \mathbb Z$ and let $S\subseteq \mathbb Z$ be a $(\delta,2)$-recurrent set. Let $(X,\mathcal B,\mu,T)$ be an MPS, with $A\subseteq X$ having $\mu(A)>\delta$.  Consider the MPS $(X,\mathcal B,\mu,T^m)$.  Since $\mu(A)>\delta$, there exists $n\in S$ such that $\mu(A\cap (T^{m})^{-n}A \cap (T^{m})^{-2n}A)>0$, meaning $\mu(A\cap T^{-mn}A\cap T^{-2(mn)}A)>0$.  Since $mn\in mS$, this proves $mS$ is $(\delta,2)$-recurrent.
\end{proof}

Our proof of Lemma \ref{lem:FinitePiece} uses the following compactness properties for recurrence.

\begin{lemma}\label{lem:RecurrenceCompactness}
  Let $k\in \mathbb N$ and $\delta\geq 0$.  If $\delta'>\delta$ and every finite subset of $S$ is $(\delta',k)$-nonrecurrent, then $S$ is $(\delta,k)$-nonrecurrent.

  Consequently, if $S$ is $(\delta,k)$-recurrent, then for all $\delta'>\delta$, there is a finite $S'\subseteq S$ which is $(\delta',k)$-recurrent.
\end{lemma}

We prove Lemma \ref{lem:RecurrenceCompactness} in \S\ref{sec:Appendix}. A special case, which is easily adapted to prove the general case, appears in Chapter 2 of \cite{ForrestThesis}.

Theorem \ref{thm:Main} is proved by combining the following lemma with the others in this section.
\begin{lemma}\label{lem:FinitePiece}
	For all $\delta>0$ and $\eta<1/2$, there exists $S\subseteq \mathbb Z$ which is $(\delta,2)$-recurrent such that $S^{\wedge 2}$ is $\eta$-nonrecurrent.
\end{lemma}
	By Lemma \ref{lem:RecurrenceCompactness} we can take $S$ to be finite in Lemma \ref{lem:FinitePiece}.

Lemmas \ref{lem:BHNonrecurrent} and \ref{lem:SqrtBHisrecurrent} will prove Lemma \ref{lem:FinitePiece}; the proof of  Lemma \ref{lem:SqrtBHisrecurrent} forms the majority of this article.

\begin{proof}[Proof of Theorem \ref{thm:Main}]

Let $\delta<\delta'<\frac{1}{2}$.  We will construct an increasing sequence of finite sets $S_1\subseteq S_2\subseteq \dots$ so that $S_n$ is $(\frac{1}{n},2)$-recurrent, and $S_n^{\wedge 2}$ is $\delta'$-nonrecurrent. Setting $S:=\bigcup_{n=1}^\infty S_n$, we get that $S$ is a set of $2$-recurrence, while every finite subset of $S^{\wedge 2}$ is $\delta'$-nonrecurrent.  Lemma \ref{lem:RecurrenceCompactness} then implies $S$ is $\delta$-nonrecurrent.

To define $S_1$, we apply Lemma \ref{lem:FinitePiece} to find an $S_1\subseteq \mathbb Z$ which is $(1,2)$-recurrent, while $S_1^{\wedge 2}$ is $\delta_1$-nonrecurrent for some $\delta_1>\delta'$.  We define the remaining $S_n$ inductively:  suppose $n\in \mathbb N$ and that $S_n$ has been chosen to be $(\frac{1}{n},2)$-recurrent, while $S_n^{\wedge 2}$ is $\delta_n$-nonrecurrent for some $\delta_n>\delta'$. Let $\eta<\frac{1}{2}$ so that $2\eta\delta_n>\delta'$. We will find $S_{n+1}\supset S_n$ so that $S_{n+1}$ is $(\frac{1}{n+1},2)$-recurrent and $S_{n+1}^{\wedge 2}$ is $2\eta\delta_n$-nonrecurrent.  To do so, apply Lemma \ref{lem:FinitePiece} to find a finite $R \subseteq \mathbb Z$ which is $(\frac{1}{n+1},2)$-recurrent such that $R^{\wedge 2}$ is $\eta$-nonrecurrent.  By Lemma \ref{lem:Combine}, choose $m\in \mathbb N$ so that $(S_n^{\wedge 2})\cup m^2(R^{\wedge 2})$ is $2\eta\delta_n$-nonrecurrent.  Now $S_{n+1}:= S_n \cup mR$ is the desired set: $mR$ is $(\frac{1}{n+1},2)$-recurent, by Lemma \ref{lem:DilatesPreserveRecurrence}, while $S_{n+1}^{\wedge 2}= (S_n^{\wedge 2})\cup m^2(R^{\wedge 2})$.  Since $2\eta\delta_n > \delta'$, this completes the inductive step of the construction.
\end{proof}

\section{Approximate Hamming balls in \tpdf{$\mathbb T^r$}{Tr} and Bohr-Hamming balls in \tpdf{$\mathbb Z$}{Z} }\label{sec:Definitions}
Let $\mathbb T$ denote the group $\mathbb R/\mathbb Z$ with the usual topology.
For $x\in \mathbb T$, let $\tilde{x}$ denote the unique element of $[0,1)$ such that $x = \tilde{x}+\mathbb Z,$ and define $\|x\|:=\min\{|\tilde{x}-n|: n\in \mathbb Z\}$.  For $r\in \mathbb N$ and $\mb x = (x_1,\dots, x_r)\in \mathbb T^r$, let $\|\mb x\|:=\max_{j\leq r} \|x_j\|$.

For $\varepsilon>0$, $r\in \mathbb N$, and $\mb x=(x_1,\dots,x_r)\in \mathbb T^r$, let
\[
w_\varepsilon(\mb x):= |\{j: \|x_j\|\geq \varepsilon\}|.
\]
So $w_\varepsilon(\mb x)$ is the number of coordinates of $\mb x$ differing from $0$ by at least $\varepsilon$.

\begin{definition}\label{def:ApproxHamm}
  For $k< r\in \mathbb N$, $\mb y\in \mathbb T^r$, and $\varepsilon>0$, we define the \emph{approximate Hamming ball} of radius $(k,\varepsilon)$ around $\mb y$ as
  \[
  \Hamm(\mb y; k,\varepsilon):=\{\mb x\in \mathbb T^r: w_{\varepsilon}(\mb y - \mb x)\leq k\}.
  \]
So $\Hamm(\mb y; k,\varepsilon)$ is the set of $\mb x=(x_1,\dots,x_r)\in \mathbb T^r$ where at most $k$ coordinates $x_i$ differ from $y_i$ by at least $\varepsilon$.
\end{definition}

If $Z$ is a topological abelian group, we say $\alpha\in Z$ \emph{generates} $Z$ if the cyclic subgroup $\{n\alpha:n\in \mathbb Z\}$ is dense in $Z$.  In other words, $\alpha$ generates $Z$ if $Z$ is the smallest closed subgroup containing $\alpha$.

The \emph{group rotation system} $(Z,\mathcal B, m_Z,R_{\alpha})$, where $\mathcal B$ is the Borel $\sigma$-algebra on $Z$ and $m_Z$ is Haar measure on $Z$, is given by $R_{\alpha}z=z+\alpha$.

\begin{definition}\label{def:BH}
  If $U=\Hamm(\mb y; k,\varepsilon) \subseteq \mathbb T^r$ is an approximate Hamming ball and $\bm\beta\in\mathbb T^r$, the corresponding \emph{Bohr-Hamming ball} of \emph{radius $(k,\varepsilon)$} is
  \[
    BH(\bm\beta,\mb y;k,\varepsilon):=\{n\in \mathbb Z:n\bm\beta\in U\}.
  \]
  If $\bm\beta$ generates $\mathbb T^r$, we say that the corresponding Bohr-Hamming ball is \emph{proper}.
\end{definition}
We write $m$ for Haar probability measure on $\mathbb T^r$.  Lemmas \ref{lem:ApproxHammNonrecurrent} and \ref{lem:BHNonrecurrent} here are implicit in \cite{Kriz} and proved explicitly in \cite{griesmer2020separating}.

\begin{lemma}\label{lem:ApproxHammNonrecurrent}
  Let $k\in \mathbb N$ and $\eta<\frac{1}{2}$. For all sufficiently large $r\in \mathbb N$ there is an $\varepsilon>0$ and  $E\subseteq \mathbb T^r$ with $m(E)>\eta$ such that $E\cap (E+U)=\varnothing$, where $U=\Hamm(\mb y;k,\varepsilon)$, with $\mb y = (\frac{1}{2},\dots, \frac{1}{2})\in \mathbb T^r$.
\end{lemma}

Lemma \ref{lem:ApproxHammNonrecurrent} is a consequence of Lemma 7.1 of \cite{griesmer2020separating}.  To derive the former from the latter, note that   \cite[Lemma 7.1]{griesmer2020separating} (in the case $p=2$ there) provides sets $E$, $E'\subseteq \mathbb T^r$ with $\mu(E)>\eta$, an approximate Hamming ball $U$ around $0_{\mathbb T^r}$ with radius $(k,\varepsilon)$ for some $\varepsilon>0$, such that $E+U\subseteq E'$ and $E'+ (\frac{1}{2},\dots,\frac{1}{2})$ is disjoint from $E'$.

\begin{lemma}\label{lem:BHNonrecurrent}
  Let $k\in \mathbb N$ and $\eta<\frac{1}{2}$. For all sufficiently large $r\in \mathbb N$, there is an $\varepsilon>0$ such that for all $\bm\beta\in \mathbb T^r$, the Bohr-Hamming ball $BH(\bm\beta,\mb y;k,\varepsilon)$ is $\eta$-nonrecurrent, where $\mb y = (\frac{1}{2},\dots,\frac{1}{2})\in \mathbb T^r$.
\end{lemma}

\begin{proof} Let $\eta<\frac{1}{2}$ and choose $r$ large enough to find the $E$ and $U$ provided by Lemma \ref{lem:ApproxHammNonrecurrent}, with $m(E)>\eta$. Let $(X,\mathcal B,\mu,T) = (\mathbb T^r,\mathcal B,m,R_{\bm\beta})$ be the group rotation on $\mathbb T^r$ determined by $\bm\beta$.  For $n\in BH(\bm\beta,\mb y;k,\eta)$, we have $R_{\bm\beta}^n E \subseteq E+U$, so $E\cap R_{\bm\beta}^n E=\varnothing.$  Since $R_{\bm\beta}$ is invertible, this means $E\cap R_{\bm\beta}^{-n}E =\varnothing$, as well.
\end{proof}

For $S\subseteq \mathbb Z$, let $\sqrt{S}:=\{n\in \mathbb Z:n^2 \in S\}$.

\begin{lemma}\label{lem:SqrtBHisrecurrent}
 For all $\delta>0$ there exists $k_0\in \mathbb N$ such that for every $r\in \mathbb N$, every proper Bohr-Hamming ball $BH:=BH(\bm\beta,\mb y; k, \varepsilon)$ with $k\geq k_0$, $\varepsilon>0$, and $\mb y\in \mathbb T^r$, $\sqrt{BH}$ is $(\delta,2)$-recurrent.
\end{lemma}

Lemma \ref{lem:SqrtBHisrecurrent} is proved using multiple ergodic averages and  characteristic factors.  The main argument is given in \S\ref{sec:MultipleErgodic}, using several reductions developed in \S\ref{sec:MultipleErgodic}-\S\ref{sec:MainProof}.

\begin{proof}[Proof of Lemma \ref{lem:FinitePiece}]
  Let $\delta>0$ and $\eta<\frac{1}{2}$.  Choose $k$  large enough to satisfy the conclusion of Lemma \ref{lem:SqrtBHisrecurrent}.  With this $k$, choose $r>k$ and $\varepsilon$ small enough to satisfy the conclusion of Lemma \ref{lem:BHNonrecurrent}.  Let $\bm\beta\in\mathbb T^r$ be generating and let $BH=BH(\bm\beta,\mb y;k,\varepsilon)$, where $\mb y=(\frac{1}{2},\dots,\frac{1}{2})\in\mathbb T^r$, so that $BH$ is $\eta$-nonrecurrent.  Finally, let $S=\sqrt{BH}$, so that $S$ is $(\delta,2)$-recurrent, by Lemma \ref{lem:SqrtBHisrecurrent}.  Since $S^{\wedge 2}\subseteq BH$, we get that $S^{\wedge 2}$ is $\eta$-nonrecurrent, as desired.
\end{proof}

\subsection{Cylinders and Fourier coefficients}

Here we define constituents of approximate Hamming balls.
\begin{definition}\label{def:Cylinder}
Given $r\in \mathbb N$, $I\subseteq \{1,\dots,r\}$, $\eta>0$, and $\mb y\in \mathbb T^r$, define the \emph{$\eta$-cylinder determined by $I$ around $\mb y$} to be \[
V_{I,\mb y,\eta}:=\{\mb x\in \mathbb T^r : \|x_i-y_i\|<\eta \text{ for all } i \in I\},\]
so that
\begin{equation}\label{eqn:UunionV}
U:=\Hamm(\mb y;k,\eta) = \bigcup_{\substack{I\subseteq \{1,\dots, r\}\\ |I| = r-k}} V_{I,\mb y, \eta}.
\end{equation} We say that $g:\mathbb T\to \mathbb R$ is a \emph{cylinder function subordinate to $U$} if $g=\frac{1}{m(V)}1_V$, where $V$ is one of the cylinders $V_{I,\mb y,\eta}$ in (\ref{eqn:UunionV}). Note that each cylinder function subordinate to $U$ is supported on $U$.
\end{definition}

  Let $\mathcal S^1$ denote the circle group $\{z\in \mathbb C:|z|=1\}$ with the usual topology and the group operation of complex multiplication. If $Z$ is a compact abelian group with Haar probability measure $m$, $\widehat{Z}$ denotes its Pontryagin dual, meaning $\widehat{Z}$ is the group of continuous homomorphisms $\chi:Z\to \mathcal S^1$; such homomorphisms are called \emph{characters} of $Z$.  Given $f:Z\to \mathbb C$, its \emph{Fourier transform} is $\hat{f}:\widehat{Z}\to \mathbb C$ given by $\hat{f}(\chi)=\int f \overline{\chi}\, dm$.

   For $s\in Z$, let $f_s$ be the translate of $f$ defined by $f_s(x):=f(x-s)$.  Then $\widehat{f_s}(\chi)=\chi(s)\hat{f}(\chi)$ for each $\chi\in \widehat{Z}$.

  As usual, for $f, g:\mathbb Z\to \mathbb C$, $f*g$ denotes their \emph{convolution}, defined as $f*g(x):=\int f(t)g(x-t)\, dm(t)$.  We will use the standard identity $\widehat{f*g}=\hat{f}\hat{g}$ (the Fourier transform turns convolution into pointwise multiplication).

Letting $\|f\|:=\bigl(\int |f|^2\, dm\bigr)^{1/2}$ denote the $L^2(m)$ norm of $f$, we have the standard Plancherel identity (\ref{eqn:Plancherel}), which leads to the subsequent lemma.
 \begin{equation}\label{eqn:Plancherel}
 	\sum_{\chi\in \widehat{Z}} |\hat{f}(\chi)|^2 = \|f\|^2.
 \end{equation}

\begin{lemma}\label{lem:NumberOfLargeCoefficients}
	Let $Z$ be a compact abelian group with Haar probability measure $m$ and $f\in L^2(m)$.  If $\|f\|\leq 1$ and $|\hat{f}(\chi_1)|,\dots, |\hat{f}(\chi_k)|$ are the $k$ largest values of $|\hat{f}|$, then $|\hat{f}(\chi)|< k^{-\frac{1}{2}}$ for all $\chi \in \widehat{Z}\setminus\{\chi_1,\dots,\chi_k\}$.
	
\end{lemma}

\begin{proof}
Let $S_1 = \{\chi_1,\dots, \chi_k\}$ be the set of characters attaining the $k$ largest values of $|\hat{f}|$, let $S_2 = \widehat{Z}\setminus S_1$, and let $c=\max\{|\hat{f}(\chi)|:\chi\in S_2\}$.  By definition, we have $|\hat{f}(\chi)|\geq c$ for all $\chi \in S_1$.

We split the left hand side of (\ref{eqn:Plancherel}) into sums over $\chi \in S_1$ and $\chi \in S_2$, then subtract the sum over $S_1$ to get
\[\sum_{\chi \in S_2} |\hat{f}(\chi)|^2 = \|f\|^2 - \sum_{\chi\in S_1} |\hat{f}(\chi)|^2.\]  Since $|\hat{f}(\chi)|\geq c$ for all $\chi \in S_1$, the right hand side is bounded above by $\|f\|^2 - kc^2$.  Since $c\leq |\hat{f}(\chi)|$ for at least one $\chi \in S_2$, the left hand side above is bounded below by $c^2$.  So
\[
c^2\leq \sum_{\chi \in S_2} |\hat{f}(\chi)|^2 = \|f\|^2 - \sum_{\chi\in S_1} |\hat{f}(\chi)|^2 \leq 1-kc^2,
\] 	
which implies $c^2\leq 1-kc^2$.  Solving, we get $c\leq (1+k)^{-\frac{1}{2}}$. This means $|\hat{f}(\chi)|< k^{-1/2}$ for all $\chi\in S_2$.
\end{proof}
\begin{remark}
  The exact form of the inequality in Lemma \ref{lem:NumberOfLargeCoefficients} is not important; we only need $\sup_{\chi\in \widehat{Z}\setminus \{\chi_1,\dots,\chi_k\}} |\hat{f}(\chi)|\leq c(k)$, where $c(k)\to 0$ as $k\to\infty$, uniformly for $\|f\|\leq 1$.
\end{remark}

Much of the proof of Lemma \ref{lem:SqrtBHisrecurrent} is contained in Lemma \ref{lem:VAnnihilates}.  The actual application requires a technical generalization (Lemma \ref{lem:BHconvolvesToUniform}).

\begin{lemma}\label{lem:VAnnihilates} Fix $k<r\in \mathbb N$, and let $U\subseteq \mathbb T^r$ be an approximate Hamming ball of radius $(k,\eta)$ with $\eta>0$.
 \begin{enumerate}
 	\item[(i)]  Let $\chi_1,\dots,\chi_k\in \widehat{\mathbb T}^r$ be nontrivial.   Then there is a cylinder function $g$ subordinate to $U$ such that for all $s\in \mathbb T^r$, we have
\[
  	\widehat{{g}_{s}}(\chi_j)=0 \qquad \text{for each } j\leq k.
\]
\item[(ii)] If $f\in L^2(m_{\mathbb T^r})$ with $\|f\|\leq 1$, there is a cylinder function $g$ subordinate to $U$ so that
\[
	|\widehat{f*g}(\chi)|< k^{-1/2} \qquad \text{for all } \chi\in\widehat{\mathbb T}^d.
\]
\end{enumerate}
\end{lemma}

\begin{proof} (i)	Assuming $k$, $r$, $\chi_j$, and $U$ are as in the statement, we may write $\chi_j$ as
\begin{equation}\label{eqn:ExplicitCharacter}
\chi_j(x_1,\dots,x_r)=e\Bigl( \sum_{l=1}^{r} n^{(j)}_lx_l\Bigr),
\end{equation}
where $e(t):=\exp(2\pi i t)$ and $n^{(j)}_l\in \mathbb Z$. Nontriviality of $\chi_j$ means that for each $j$, at least one of the $n^{(j)}_l$ is nonzero.  So choose one such index $l_j$ for each $j\leq k$ and let $I=\{1,\dots,r\}\setminus \{l_1,\dots,l_k\}$.  In case some of the $l_j$ repeat, remove additional elements from $I$ so that $|I|=r-k$.

Writing $U$ as $\Hamm(\mb y;k,\eta)$, let $V = V_{I,\mb y,\eta}=\{\mb x\in \mathbb T^d:\|y_l-x_l\|<\eta \text{ for all } l \in I\}$, so that $V\subseteq U$. Let $g:=\frac{1}{m(V)}1_V$, so that $g$ is a cylinder function subordinate to $U$, and let $j\leq k$.  To prove that $\hat{g}(\chi_j)=0$, note that $g$ does not depend on any of the coordinates $x_{l_j}$, so we can simplify the right hand side of (\ref{eqn:ExplicitCharacter})  as $e\Bigl(\sum_{\substack{l=1 \\ l\neq l_j}}^{r} n_{l}^{(j)}x_l\Bigr)e(n_{l_j}^{(j)} x_{l_j})$ and write $\hat{g}(\chi_j)=\int g \overline{\chi}_j \, dm$ as
\[
\int_{\mathbb T^{r-1}} g(x_1,\dots,x_r)   e\Bigl(-\sum_{\substack{l=1 \\ l\neq l_j}}^{r} n_{l}^{(j)}x_l\Bigr)\, dx_1\dots \, dx_{l_{j-1}}\, dx_{l_{j+1}}\, \dots \, dx_{r} \, \int_{\mathbb T} e(-n_{l_j}^{(j)} x_{l_j})\, dx_{l_j}.
\]
Since $\int e(-n_{l_j}^{(j)} x_{l_j})\, dx_{l_j}=0$, we conclude that $\hat{g}(\chi_j)=0$ for each $j$.  To complete the proof of Part (i), we observe that $\widehat{g_{s}}(\chi)=\chi(s)\hat{g}(\chi)$ for each $\chi$.

To prove (ii), assume $f:\mathbb T^r\to \mathbb C$ has $\|f\|\leq 1$, and let $|\hat{f}(\chi_1)|, \dots, |\hat{f}(\chi_k)|$ be the $k$ largest values of $|\hat{f}|$. By Part (i), choose a cylinder function $g$ subordinate to $U$ so that $\hat{g}(\chi_j)=0$ for these $\chi_j$. Then $|\hat{f}(\chi)|< k^{-1/2}$ for all other $\chi$, by Lemma \ref{lem:NumberOfLargeCoefficients}.   Note that $|\hat{g}(\chi)|\leq 1$ for all $\chi\in \widehat{\mathbb T}^d$, since $\int |g|\, dm =1$.     We therefore have $\widehat{f*g}(\chi_j)=\hat{f}(\chi_j)\hat{g}(\chi_j)=0$ for $j=1,\dots,k$, while $|\widehat{f*g}(\chi)|\leq |\hat{f}(\chi)|<k^{-1/2}$ for all other $\chi$. \end{proof}

\section{Multiple ergodic averages}\label{sec:MultipleErgodic}
Some of our reductions use facts from the general theory of nilsystems, mainly contained in \cite{FrantzikinakisThreePoly} and \cite{FrKrPolyAffine}.  Readers who want a general introduction to the theory can consult \cite{HostKraBook}.

If $(X,\mathcal B,\mu,T)$ is an MPS and $f$ is a bounded function on $X$, let
\[
L_3(f,T):=\lim_{N\to\infty} \frac{1}{N}\sum_{n=1}^N \int f\cdot T^n f\cdot T^{2n}f\, d\mu.
\]
The existence of this limit was established in \S 3 of \cite{F77}.

In this section we prove Lemma \ref{lem:SqrtBHisrecurrent} using Lemma \ref{lem:BHKroneckerCharacteristic}, which estimates variants of $L_3(f,T)$.  In \S\ref{sec:Reformulation} we state a more convenient form of Lemma \ref{lem:BHKroneckerCharacteristic} and outline its proof.

We will use the following known result, which follows by combining a special case of Theorem 2.1 of \cite{BHMP} with the multidimensional Szemer\'edi Theorem \cite{FK79}.

\begin{theorem}\label{thm:UniformSz}
	For all $\delta>0$, there exists $c(\delta)>0$ such that for every MPS $(X,\mathcal B,\mu,T)$ and every $f:X\to [0,1]$ with $\int f\, d\mu> \delta$ we have
	\begin{equation}\label{eqn:UniformRoth}
		L_3(f,T) > c(\delta).
	\end{equation}
\end{theorem}

\begin{definition}
  We say that $\mb X=(X,\mathcal B,\mu,T)$ is \emph{ergodic} if $\mu(A\triangle T^{-1}A)=0$ implies $\mu(A)=0$ or $\mu(A)=1$ for every $A\in \mathcal B$. We say that $\mb X$ is \emph{totally ergodic} if for every $m\in \mathbb N$, the system $(X,\mathcal B,\mu,T^m)$ is ergodic.
\end{definition}

\begin{remark}\label{rem:Standard}
  When determining whether a set is a set of $k$-recurrence, we may restrict our attention to \emph{ergodic} MPSs where $\mu$ is a regular Borel measure on a compact metric space $X$; cf.~\S\S 7.2.2-7.2.3 of \cite{EinsiedlerWard}.
\end{remark}

When we say a sequence $(b_n)_{n\in\mathbb N}$ of natural numbers has \emph{linear growth}, we mean that it is strictly increasing and $\limsup_{n\to\infty} b_n/n < \infty$. Note that a strictly increasing sequence has linear growth if and only if the set of terms $B=\{b_n:n\in \mathbb N\}$ satisfies $\underbar{d}(B):=\liminf_{n\to\infty} \frac{|B\cap \{1,\dots,n\}|}{n}>0$. Enumerating the positive elements of $\sqrt{BH}$ in increasing order, where $BH$ is a proper Bohr-Hamming ball always results in a sequence of linear growth.  To see this, write $\sqrt{BH}$ as $\{n\in \mathbb Z:n^2\bm\beta\in U\}$ for some approximate Hamming ball $U\subseteq \mathbb T^r$ and generator $\bm\beta\in \mathbb T^r$.  Then \[\lim_{N\to\infty}\frac{|\sqrt{BH} \cap [1,\dots N]|}{N} = \lim_{N\to\infty} \frac{1}{N}\sum_{n=1}^N 1_U(n^2\bm\beta) = m(U),\] by Weyl's theorem on uniform distribution of polynomials (see Lemma \ref{lem:ConnectedWeyl} below).  Since $n = |\sqrt{BH}\cap [1,\dots,b_n]|$, this implies $b_n/n$ is bounded. Likewise, if $g$ is a cylinder function subordinate to $U$ (Definition \ref{def:Cylinder}), then enumerating  $\{n\in \mathbb N: g(n^2\bm\beta)>0\}$ in increasing order results in a sequence of linear growth.

The next lemma says that $L_3(f,T)$ can be approximated by averaging over elements of $\sqrt{BH}$, provided $\mb X$ is \emph{totally ergodic} and $BH$ is a proper Bohr-Hamming ball of radius $(k,\eta)$ with $k$ sufficiently large.  In passing to the general case, we need to consider $\sqrt{BH}/\ell:=\{n\in \mathbb Z: \ell n \in \sqrt{BH}\}$.

\begin{lemma}\label{lem:BHKroneckerCharacteristic}
  For all $\varepsilon>0$, there is a $k\in \mathbb N$ such that for every totally ergodic MPS $(X,\mathcal B,\mu,T)$, every $f:X\to [0,1]$, every proper Bohr-Hamming ball $BH$ of radius $(k,\eta)$ ($\eta>0$), and all $\ell\in \mathbb N$, there is a sequence $b_n\in \sqrt{BH}/\ell$  having linear growth such that
  \begin{equation}\label{eqn:RothApprox}
  \lim_{N\to \infty} \Bigl|\frac{1}{N}\sum_{n=1}^N \int f\cdot T^{b_n}f \cdot T^{2b_n}f\, d\mu - L_3(f,T)\Bigr|<\varepsilon \|f\|^2.
  \end{equation}
  Consequently, if $\int f\, d\mu>\delta$ and $k$ is sufficiently large (depending only on $\delta$), we have
  \begin{equation}\label{eqn:RothAlongSqrtBH}
  \lim_{N\to\infty} \frac{1}{N}\sum_{n=1}^N \int f\cdot T^{b_n}f \cdot T^{2b_n}f\, d\mu > c(\delta)/2,
  \end{equation}
  where $c(\delta)$ is defined in Theorem \ref{thm:UniformSz}.
\end{lemma}

Lemma \ref{lem:BHKroneckerAlternative} is a convenient reformulation of Lemma \ref{lem:BHKroneckerCharacteristic}. In \S\ref{sec:Reformulation} we outline its proof, which occupies the majority of this article.

\begin{remark}
  We do not know whether the condition ``totally ergodic'' can be replaced with ``ergodic'' in Lemma \ref{lem:BHKroneckerCharacteristic}.  The main obstruction to this replacement is our lack of a convenient representation of ergodic, but not totally ergodic,  2-step affine nilsystems.
\end{remark}

\subsection{Factors and Extensions}\label{sec:FactorExtension}

If $\mb X = (X,\mathcal B,\mu,T)$ and $\mb Y = (Y,\mathcal D,\nu,S)$ are MPSs, we say that $\mb Y$ is a \emph{factor} of $\mb X$ if there is a measurable $\pi: X\to Y$ intertwining $S$ and $T$, meaning
\[
\pi(Tx) = S\pi(x) \qquad \text{for } \mu\text{-a.e.~} x\in X,
\]
and $\mu(\pi^{-1}D) = \nu(D)$ for all $D\in \mathcal D$. Strictly speaking, the factor is the \emph{pair} $(\pi, \mb Y)$, and we refer to ``the factor $\pi:\mb X\to \mb Y$''.

If $\pi:\mb X\to \mb Y$ is a factor and $f\in L^2(\mu)$ is equal $\mu$-almost everywhere to a function of the form $g\circ \pi$, with $g\in L^2(\nu)$, we say that $f$ is \emph{$\mb Y$-measurable}.  This is equivalent to saying that $f$ is $\pi^{-1}(\mathcal D)$-measurable (modulo $\mu$).  We denote by $P_{\mb Y}$ the orthogonal projection from $L^2(\mu)$ to the space of $\pi^{-1}(\mathcal D)$-measurable functions.  Given $f\in L^2(\mu)$, we identify $P_{\mb Y}f$ with $\tilde{f}\in L^2(\nu)$ satisfying $P_{\mb Y} f = \tilde{f}\circ \pi$.

We repeatedly use, without comment, the fact that $P_{\mb Y}$ is a positive operator preserving integration with respect to $\mu$.  In other words, if $f(x)\geq 0$ for $\mu$-a.e.~$x$, then $P_{\mb Y}f(x)\geq 0$ for $\mu$-a.e.~$x$, and $\int f\, d\mu = \int P_{\mb Y} f\, d\mu$. Consequently, $\sup f \geq \tilde{f}(y)\geq \inf f$ for $\nu$-a.e.~$y$ and $\int \tilde{f}\, d\nu = \int f\, d\mu$.

\begin{remark}\label{rem:Extensions}
When $\pi: \mb X\to \mb Y$ is a factor, we say that $\mb X$ is an \emph{extension} of $\mb Y$.  If we wish to prove an inequality on ergodic averages for a system $\mb Y$, it suffices to prove that inequality for an extension $\pi:\mb X\to \mb Y$, since the integrals $\int f_0\cdot S^af_1\cdot S^{b}f_2\, d\nu$ can be written as $\int h_0\cdot T^a h_1\cdot T^{b}h_2\, d\mu$, where $h_i = f_i\circ \pi$.  This observation will be used in \S\ref{sec:MainProof}.
\end{remark}
\subsection{Reducing to total ergodicity}
The next lemma is used to deduce Lemma \ref{lem:SqrtBHisrecurrent} from  Lemma \ref{lem:BHKroneckerCharacteristic} and Theorem \ref{thm:UniformSz}.  Part (i) is a special case of Corollary 4.6 in \cite{BHK}, and Part (ii) is an immediate consequence of Part (i).  Here ``$\mb Y$ is an inverse limit of ergodic nilsystems'' means that for all $f\in L^\infty(\nu)$ and $\varepsilon>0$, there is an a factor $\pi:\mb Y\to \mb Z$, where $\mb Z=(Z,\mathcal Z,\eta,R)$ is an ergodic nilsystem and $\|f-P_{\mb Z}f\|_{L^1(\nu)}<\varepsilon$.

\begin{lemma}\label{lem:NilCharacteristic}
Let $\mb X=(X,\mathcal B,\mu,T)$ be an ergodic measure preserving system. There is a factor $\pi:\mb X\to \mb Y=(Y,\mathcal D,\nu,S)$ which is an inverse limit of ergodic nilsystems such that

\begin{enumerate}
\item[(i)] for all $f_i\in L^\infty(\mu)$, letting $\tilde{f}_i\circ \pi =P_{\mb Y}f_i$, we have
\[
\lim_{N\to \infty} \frac{1}{N}\sum_{n=1}^N \Bigl| \int f_0\cdot T^n f_1\cdot T^{2n}f_2 \, d\mu - \int \tilde{f}_0\cdot S^n \tilde{f}_1\cdot S^{2n}\tilde f_2\, d\nu\Bigr|=0.
\]

\item[(ii)]  If $(b_n)_{n\in \mathbb N}$ is a sequence of linear growth, then
\[
\lim_{N\to\infty} \frac{1}{N}\sum_{n=1}^N \Bigl|\int f_0\cdot T^{b_n}f_1 \cdot T^{2b_n}f_2 \, d\mu - \int \tilde{f}_0\cdot S^{b_n}\tilde{f}_1 \cdot S^{2b_n}\tilde{f}_2\, d\nu\Bigr|=0.
\]
\end{enumerate}
\end{lemma}
To derive Part (ii) from Part (i),  note that
\begin{align*}
\Bigl|\frac{1}{N}\sum_{n=1}^N \int f\cdot T^{b_n}f \cdot T^{2b_n}f& \, d\mu - \int \tilde{f}\cdot S^{b_n}\tilde{f} \cdot S^{2b_n}\tilde{f}\, d\nu
\Bigr| \\ &\leq \frac{b_N}{N} \cdot \frac{1}{b_N}\sum_{n=1}^{b_N}\Bigl| \int f_0\cdot T^n f_1\cdot T^{2n}f_2 \, d\mu - \int \tilde{f}_0\cdot S^n \tilde{f}_1\cdot S^{2n}\tilde f_2\, d\nu\Bigr|\\
&\underset{N\to\infty}{\longrightarrow} 0,
\end{align*}
since $\frac{b_N}{N}$ is bounded.

We get the next result by combining the definition of ``inverse limit'' with the fact that for every ergodic nilsystem $(Y,\mathcal D,\nu,S)$ there is an $\ell \in \mathbb N$ such that the ergodic components of $(Y,\mathcal D,\nu,S^\ell)$ are totally ergodic; see  \cite[Proposition 2.1]{FrantzikinakisThreePoly} for justification.

\begin{lemma}\label{lem:TotallyErgodicFactor}
If $(X,\mathcal B,\mu,T)$ is an inverse limit of ergodic nilsystems, $f:X\to [0,1]$, and $\varepsilon>0$, there is a factor $\mb Y = (Y,\mathcal D,\nu,S)$ and $\ell \in \mathbb N$ such that
\begin{enumerate}
	\item[(i)] $\|f-P_{\mb Y}f\|_{L^1(\mu)}<\varepsilon$,
	\item[(ii)] the ergodic components of $(Y,\mathcal D,\nu,S^\ell)$ are totally ergodic.
\end{enumerate}
\end{lemma}

\begin{notation}\label{not:Ycomponents} When $Y$ is the phase space of an ergodic nilsystem where $(Y,\mathcal D,\nu,S^\ell)$ is totally ergodic, we will enumerate its connected components as $Y_1,\dots,Y_M$, and write $\nu_i:=\frac{1}{M}\nu|_{Y_i}$.  Each $\mb Y_i:=(Y_i,\mathcal D_i,\nu_i,S^{\ell})$ is an ergodic component of $(Y,\mathcal D,\nu_Y,S^\ell)$.  If $\mb X$ is an extension of $\mb Y$ with factor map $\pi:X\to Y$, we let $X_i=\pi^{-1}(Y_i)$, $\mb \mu_i:= \frac{1}{M}\mu|_{X_i}$, $\mathcal B_i:=\{B\cap X_i:B\in \mathcal B\}$, and $\mb X_i=(X_i,\mathcal B_i,\mu_i,T^{\ell})$.  It is easy to verify that $\mb Y_i$ is a factor of $\mb X_i$ with factor map $\pi|_{X_i}$.
\end{notation}

\begin{remark}  Here we identify a technical difficulty common in multiple recurrence arguments.  Readers familiar with the use of Markov's inequality to overcome this difficulty may skip to the proof of Lemma \ref{lem:SqrtBHisrecurrent}.

 Our proof of Lemma \ref{lem:SqrtBHisrecurrent} starts with an ergodic, but not totally ergodic, MPS $\mb X=(X,\mathcal B,\mu,T)$.  By Lemma \ref{lem:NilCharacteristic} it suffices to prove the lemma in the special case where $\mb X$ is an inverse limit of ergodic nilsystems, so we assume $\mb X$ is such an inverse limit.  We then consider $f:X\to [0,1]$ with $\int f\, d\mu>\delta$.  The goal is to find an $\ell\in \mathbb N$ and a sequence $(b_n)$ of elements of $\sqrt{BH}/\ell$ satisfying inequality (\ref{eqn:PositiveIntegral}) below.  The main difficulty arises when trying to exploit the structure of nilsystems: Lemma \ref{lem:BHKroneckerCharacteristic} requires total ergodicity, so we fix $\varepsilon>0$ and choose a factor $\pi :\mb X\to \mb Y$ where $\mb Y$ is an ergodic nilsystem satisfying (i) and (ii) in Lemma \ref{lem:TotallyErgodicFactor}. We choose $\ell$ so that the ergodic components of $(Y,\mathcal D,\nu,S^\ell)$ are totally ergodic, and we enumerate these components as $\mb Y_i = (Y_i, \mathcal D_i, \nu_i,S^\ell)$, $i=1,\dots,M$.  With Notation \ref{not:Ycomponents} defined above, let $\tilde{f}\circ \pi=P_{\mb Y}f$ and $\tilde{f}_i=\tilde{f}|_{Y_i}$.  Lemma \ref{lem:BHKroneckerCharacteristic} allows us to choose, \emph{for each ergodic component $\mb Y_i$ where} $\int \tilde{f}_i\, d\nu_i>\delta/2$, a sequence $b_n^{(i)}\in \sqrt{BH}/\ell$ having linear growth, such that
\begin{equation}\label{eqn:GoodForZi}\lim_{N\to\infty} \frac{1}{N}\sum_{n=1}^N \int \tilde{f}_i\cdot S^{\ell b_n^{(i)}} \tilde{f}_i \cdot S^{2\ell b_n^{(i)}}\tilde{f}_i\, d\nu_i>c(\delta/2)/2.
\end{equation}  The choice of $b_{n}^{(i)}$ depends on $\mb Y_i$, so  (\ref{eqn:GoodForZi}) implies only that
\begin{equation}\label{eqn:oneOverM}
\lim_{N\to\infty} \frac{1}{N}\sum_{n=1}^N \int \tilde{f}\cdot S^{\ell b_n^{(i)}} \tilde{f} \cdot S^{2\ell b_n^{(i)}}\tilde{f}\, d\nu> \frac{1}{M}\frac{c(\delta/2)}{2}.
\end{equation}  If $M$ is large then $\|f-P_{\mb Y}f\|_{L^1(\mu)}$ may be large compared to $\frac{1}{M}\frac{c(\delta/2)}{2}$, and  (\ref{eqn:oneOverM}) will not immediately imply (\ref{eqn:PositiveIntegral}).  To overcome this obstacle, we want to find an $i$ where (\ref{eqn:GoodForZi}) holds and $\frac{1}{M}\|f_i-P_{\mb Y}f_i\|_{L^1(\mu)}$ is sufficiently small to make $\int \tilde{f}_i\cdot S^{\ell a}\tilde{f}_i \cdot S^{\ell b}\tilde{f}_i\, d\nu_i$ close to $\int f_i\cdot T^{\ell a}f_i\cdot T^{\ell b} f_i\, d\mu_i$ for all $a, b$.  Such an $i$ is provided by two straightforward applications of Markov's inequality outlined in \S\ref{sec:Markov}.
\end{remark}

Before proving Lemma \ref{lem:SqrtBHisrecurrent} we recall its statement:  for all $\delta>0$, there is a $k_0\in \mathbb N$ such that for every proper Bohr-Hamming ball $BH:=BH(\bm\beta,\bm y; k, \eta)$ with $k\geq k_0$, $\eta>0$, and $\mb y\in \mathbb T^r$, $\sqrt{BH}$ is $(\delta,2)$-recurrent.

\begin{proof}[Proof of Lemma \ref{lem:SqrtBHisrecurrent}, assuming Lemma \ref{lem:BHKroneckerCharacteristic}]  Let $\delta>0$, and choose $k_0\in \mathbb N$ so that for all $k\geq k_0$, inequality (\ref{eqn:RothAlongSqrtBH}) holds in Lemma \ref{lem:BHKroneckerCharacteristic}  with $c(\delta/2)$ in place of $c(\delta)$. Let $BH$ be a proper Bohr-Hamming ball with radius $(k,\eta)$ for some $\eta>0$. It suffices to prove that for every MPS $(X,\mathcal B,\mu,T)$ with $A\subseteq X$ having $\mu(A)>\delta$
	\begin{equation}\label{eqn:PositiveIntersection}
	\mu(A\cap T^{-n}A\cap T^{-2n}A) > 0 \qquad \text{for some } n\in \sqrt{BH}.
	\end{equation}
By Remark \ref{rem:Standard}, we need only consider \emph{ergodic} MPSs.  We will prove that if $\mb X$ is ergodic and $f:X\to [0,1]$ has $\int f\, d\mu>\delta$, then there is a sequence of elements $b_n\in \sqrt{BH}$ with linear growth such that
\begin{equation}\label{eqn:PositiveIntegral}
\liminf_{N\to\infty} \frac{1}{N}\sum_{n=1}^N \int f\cdot T^{b_n}f\cdot T^{2b_n}f\, d\mu>0.
\end{equation}
The special case of (\ref{eqn:PositiveIntegral}) where $f=1_A$ implies (\ref{eqn:PositiveIntersection}), as the integral then simplifies to $\mu(A\cap T^{-b_n}A\cap T^{-2b_n}A)$.

By Part (ii) of Lemma \ref{lem:NilCharacteristic}, it suffices to prove (\ref{eqn:PositiveIntegral}) when $\mb X=(X,\mathcal B,\mu,T)$ is an inverse limit of ergodic nilsystems.  We now fix such an $\mb X$, and $f:X\to [0,1]$ with $\int f\, d\mu>\delta$.

	Let $\varepsilon =  \frac{\delta}{24}c\bigl(\frac{\delta}{2}\bigr)$, and let  $\pi:\mb X\to \mb Y$ be the factor provided by Lemma \ref{lem:TotallyErgodicFactor} for this $\varepsilon$, with $\ell \in \mathbb N$ chosen so that the ergodic components $\mb Y_i$ of $(Y,\mathcal D,\nu,S^{\ell})$ are totally ergodic.  Let $M$ be the number of ergodic components,\footnote{We can take $\ell = M$, but we do not need this fact.} so that $\mu(Y_i)=1/M$ for each $i$.
	
	Let $X_i=\pi^{-1}(Y_i)$ and let $f_i=1_{X_i}f$, so that the $X_i$ partition $X$ into sets of measure $1/M$, and $\sum_i \int f_i \, d\mu = \int f \, d\mu >\delta$.  Observe that $P_{\mb Y}f_i$ is supported on $X_i$ and $\int P_{\mb Y}f_i\, d\mu = \int f_i\, d\mu$ for each $i$.
	
    Setting $\mb Y_{i}:=(Y_i,\mathcal D_i,\nu_i,S^{\ell})$, where $\nu_i:=M\nu|_{Y_i}$, we get that $\mb Y$ is a totally ergodic MPS.  Likewise $\mb X_{i}:= (X_i, \mathcal B_i, \mu_i,T^{\ell})$, with $\mu_i:=M\mu|_{X_i}$ is an MPS (possibly not ergodic), with $\pi|_{X_i}:X_i\to Y_i$ a factor map.  To prove (\ref{eqn:PositiveIntegral}), we will find a sequence $b_n$ of elements of $\sqrt{BH}/\ell $ having linear growth and $i\leq M$ with
    \begin{equation}\label{eqn:BH/lGood}
     \liminf_{N\to\infty} \frac{1}{N}\sum_{n=1}^N \int f_i\cdot T^{\ell b_n}f_i\cdot T^{2\ell b_n} f_i\, d\mu> 0.
    \end{equation}
	We claim that there is an $i$ such that \begin{align}
 	\label{eqn:AiLarge}	\int f_i\, d\mu &>  \frac{\delta}{2M}, \quad\text{ and }	\\
	\label{eqn:CloseOnXi} \|f_i - P_{\mb Y}f_i\|_{L^1(\mu)}&<\frac{c(\delta/2)}{12M}.
	\end{align}
	This $i$ is provided by Lemmas \ref{lem:Markov1} and \ref{lem:Markov2}: setting
\[I:=\Bigl\{i:\int f_i\, d\mu>\frac{\delta}{2M}\Bigr\}, \qquad J:=\Bigl\{i: \int |f_i - P_{\mb Y}f_i|\, d\mu < \frac{c(\delta/2)}{12M}\Bigr\},\] we get $|I| > M\delta/2$ and $|J|> M(1-\frac{\varepsilon}{c(\delta/2)/12})=M(1-\frac{\delta}{2})$.  Thus $|I|+|J|>M$, implying $I\cap J$ is nonempty.
	
	Fix $i$ satisfying (\ref{eqn:AiLarge}) and (\ref{eqn:CloseOnXi}). Note that (\ref{eqn:AiLarge}) and the definition of $\nu_i$, $\mu_i$, and $\tilde{f}_i$ imply
	\begin{equation}\label{eqn:fiLarge}
	\int \tilde{f}_i\, d\nu_i > \delta/2.
	\end{equation}
	Since $(Y_i,\mathcal B_i,\nu_i,S^{\ell})$ is totally ergodic, we may apply Lemma \ref{lem:BHKroneckerCharacteristic} to choose a sequence of elements $b_n\in \sqrt{BH}/\ell$ having linear growth and satisfying
	\begin{equation}\label{eqn:bForFactor}
		\lim_{N\to\infty} \frac{1}{N}\sum_{n=1}^N \int \tilde{f}_i\cdot S^{\ell b_n}\tilde{f}_i \cdot S^{2\ell b_n} \tilde{f}_i\, d\nu_i > c(\delta/2)/2.
	\end{equation}
Inequality  (\ref{eqn:CloseOnXi}), the bounds $\|f_i\|_{\infty}\leq 1$, $\|P_{\mb Y}f_{i}\|_{\infty}\leq 1$, and Lemma \ref{lem:MultiBound} imply
	\begin{equation}\label{eqn:DifferenceExpand}
		\Bigl|\int f_i\cdot T^{\ell a} f_i \cdot T^{\ell b} f_i\, d\mu_i - \int P_{\mb Y_i}f_i\cdot T^{\ell a} P_{\mb Y_i}f_i \cdot T^{\ell b} P_{\mb Y_i}f_i\, d\mu_i\Bigr|< \tfrac{1}{4}c(\delta/2).
	\end{equation}
	for each $a,b\in \mathbb N$. Recalling the definition of $\mu_i$ and $\nu_i$, we see that for all sufficiently large $N$,
\begin{align*}
		 \frac{1}{N} \sum_{n=1}^{N} \int f_i\cdot T^{\ell b_n} f_i \cdot T^{2\ell b_n} f_i\, d\mu
		&> \frac{1}{N}\sum_{n=1}^{N} \int \tilde{f}_i \cdot S^{\ell b_n} f_i \cdot S^{2\ell b_n} f_i\, d\nu - \frac{c(\delta/2)}{4M}  && \text{by }  (\ref{eqn:DifferenceExpand})\\
		&> 	\frac{c(\delta/2)}{2M} - \frac{c(\delta/2)}{4M} && \text{by } (\ref{eqn:bForFactor}) \\
		&= \frac{c(\delta/2)}{4M}.
\end{align*}
The above inequalities imply (\ref{eqn:BH/lGood}). Since $f\geq f_i$ pointwise and we chose $b_n\in \sqrt{BH}/\ell$, this implies (\ref{eqn:PositiveIntegral}) and completes the proof of Lemma \ref{lem:SqrtBHisrecurrent}. \end{proof}

 \section{Reformulation of Lemma \ref{lem:BHKroneckerCharacteristic}}\label{sec:Outline}

\subsection{Reformulation}\label{sec:Reformulation}

Lemma \ref{lem:BHKroneckerCharacteristic} is an immediate consequence of the following reformulation.  This version allows us to apply the theory of characteristic factors.

\begin{lemma}\label{lem:BHKroneckerAlternative}
Let $k<r\in \mathbb N$, $\ell\in \mathbb N$, let $\bm\beta\in \mathbb T^r$ be generating, and let $U\subseteq \mathbb T^r$ be an approximate Hamming ball of radius $(k,\eta)$ for some $\eta>0$. For every totally ergodic MPS $(X,\mathcal B,\mu,T)$, and every measurable $f:X\to [0,1]$, there is a cylinder function $g=\frac{1}{m(V)}1_V$ subordinate to $U$ such that
	\begin{equation}\label{eqn:RothApproxWithg}
		\lim_{N\to \infty} \Bigl|\frac{1}{N}\sum_{n=1}^N g(n^2\ell^2\bm\beta)\int f\cdot T^{n} f \cdot T^{2n} f\, d\mu - L_3(f,T)\Bigr|<2k^{-1/2}\|f\|^2.
	\end{equation}
\end{lemma}

While $U$ does not depend on $f$ in Lemma \ref{lem:BHKroneckerAlternative}, the choice of $g$ to satisfy (\ref{eqn:RothApproxWithg}) does depend on $f$.

We prove Lemma \ref{lem:BHKroneckerAlternative} in \S\ref{sec:MainProof}. The derivation of Lemma \ref{lem:BHKroneckerCharacteristic} from Lemma \ref{lem:BHKroneckerAlternative} is an instance of the following general principle: if $a_n$ is a bounded sequence, $B\subseteq \mathbb N$ is enumerated as $\{b_1<b_2<\dots\}$, and $d(B):=\lim_{N\to\infty} \frac{|B\cap \{1,\dots,N\}|}{N}>0$, then
\[
\lim_{N\to\infty} \frac{1}{N}\sum_{n=1}^N a_{b_n} = \lim_{N\to\infty} \frac{1}{Nd(B)}\sum_{n=1}^N1_{B}(n)a_n
\]
provided the limit on the right exists.  Note that $(b_n)_{n\in \mathbb N}$ has linear growth if $d(B)>0$.

We will apply this principle with $a_n = \int f\cdot T^{n} f \cdot T^{2n} f\, d\mu$ and $B=\{n:n^2\ell^2\bm\beta\in V\}$, where $V$ is a cylinder contained in $U$.  Then $g=\frac{1}{m(V)}1_V$ is a cylinder function subordinate to $U$, and $g(n^2\ell^2\bm\beta)=\frac{1}{d(B)}1_B(n)$.  The equation $d(B)=m(V)$ follows from Weyl's theorem on uniform distribution (cf.~\S\ref{sec:Weyl}).  Note that this $B$ is contained in $\sqrt{BH}/\ell$, where $BH$ is the Bohr-Hamming ball corresponding to $U$, with frequency $\bm\beta$.

\begin{remark}
 The exact form of the bound in (\ref{eqn:RothApproxWithg}) is not important in the sequel.  The only relevant property is that the coefficient of $\|f\|^2$ tends to $0$ as $k\to \infty$.
\end{remark}

\subsection{Outline of a special case of Lemma \ref{lem:BHKroneckerAlternative}}\label{sec:OutlineSpecial}  This outline highlights  the key steps in our proof while avoiding some complications.

We begin with an arbitrary \emph{totally ergodic} measure preserving system $\mb X=(X,\mathcal B,\mu,T)$,  $f: X\to [0,1]$, and $k\in \mathbb N$. We let $r>k$, $\eta>0$, and  fix an approximate Hamming ball $U=\Hamm(\mb y;k,\eta)\subseteq \mathbb T^r$ and a generator $\bm\beta \in \mathbb T^r$.  We want to find  a cylinder function $g$ subordinate to $U$ so that
\begin{equation}\label{eqn:DefAN}
A_N(f,g):= \frac{1}{N}\sum_{n=1}^N g(n^2\bm\beta)\int f\cdot T^nf\cdot T^{2n}f\, d\mu
\end{equation}
satisfies $\lim_{N\to\infty} |A_N(f,g)-L_3(f,T)|<2k^{-1/2}\|f\|^2$.

In \S\ref{sec:Eigenvalues}-\S\ref{sec:AffineReduction} we will reduce to the case where $\mb X$ is a \emph{standard 2-step Weyl system}.  This means that $(X,\mathcal B,\mu ,T)$ can be realized with $X=\mathbb T^d\times \mathbb T^d$,  $d\in \mathbb N$, $\mu=$ Haar probability measure on $\mathbb T^d\times\mathbb T^d$, and $T$ is given by $T(x, y)=( x+\bm\alpha,y+ x)$, for some generator $\bm\alpha\in\mathbb T^d$.  The orbits of $T$ can be computed explicitly: $T^n(x,y)=(x+n\bm\alpha, y+nx+\tbinom{n}{2}\bm\alpha)$.  This reduction relies on the theory of characteristic factors, especially Theorem B of \cite{FrantzikinakisThreePoly}.

To simplify this outline, we assume $r=d$ and $\bm\beta = \bm \alpha$.  We write functions on $\mathbb T^d\times \mathbb T^d$ with variables displayed as $f(x,y)$, where $x, y\in \mathbb T^d$.  Writing $m\times m$ for Haar probability measure on $\mathbb T^d\times \mathbb T^d$,  we write $\int f\, dm\times m$  as $\int f(x,y)\, dx\, dy$, or $\int f\binom{x}{y}\, dx\, dy$ to save space.  With these assumptions, the averages in (\ref{eqn:DefAN}) become
\[
B_N
:= \frac{1}{N}\sum_{n=1}^N g(n^2\bm\alpha) \int_{\mathbb T^d\times \mathbb T^d} f\colvec{x\\y}f\colvec{x+n\bm\alpha\\y+nx+\tbinom{n}{2}\bm\alpha}f\colvec{x+2n\bm\alpha \\ y+2nx + \tbinom{2n}{2}\bm\alpha}\, dx\, dy.
\]
Proposition \ref{prop:WeakLimit} provides an explicit formula for $\lim_{N\to\infty} B_N$.  Under the present assumptions, it says
\begin{equation}\label{eqn:Blim}
	\lim_{N\to\infty} B_N = \int_{(\mathbb T^d)^4} f(x,y)f(x+s,y+t) \Bigl(\int_{\mathbb T^d} f(x+2s,y+2t+2w)g(w) \, dw\Bigr) \, ds\, dt \, dx\, dy.
\end{equation}
Write $I$ for the right hand side above, and define
\[
f*_2 g(x,y):= \int f(x,y+2w) g(w)\, dw.
\]
Using Lemma \ref{lem:BHconvolvesToUniform} (a generalization of Lemma \ref{lem:VAnnihilates}), we choose a cylinder function $g$ subordinate to $U$ such that $|\widehat{f*_2 g}(\chi,\psi)|<k^{-1/2}$ for all $(\chi,\psi)\in \widehat{\mathbb T}^d\times \widehat{\mathbb T}^d$ with $\psi$ nontrivial.
We set $f'(x):=\int f(x,y)\, dy$ and $J':= \int \int f'(x)f'(x+s)f'(x+2s)\, dx\, ds$.  By Lemma \ref{lem:SmallFourierToW}, the bound on $\widehat{f*_2 g}$ will imply
\begin{equation}\label{eqn:Clim}
	|I - J'| < k^{-1/2}\|f\|^2.
\end{equation}
We can also prove (directly, or using Theorem \ref{thm:ErgodicRoth}) that $L_3(f,T)=J'$.  Combining (\ref{eqn:Clim}) with (\ref{eqn:Blim}), we then have (\ref{eqn:RothApproxWithg}),  completing the outline of this special case. The factor $2$ on the right-hand side of (\ref{eqn:RothApproxWithg}) accounts for the reduction to Weyl systems.

In the general case we must compute $\lim_{N\to\infty} A_N(f,g)$ for $d\neq r$ and $\bm\beta\neq\bm\alpha$.  The integral $\int f(x+2s,y+2t +2w) g(w)\, dw$ in (\ref{eqn:Blim}) will then be replaced by an integral over an \emph{affine joining} of $\mathbb T^d$ with $\mathbb T^r$ (Definition \ref{def:AffineJoining}), but the computation in this case is not substantially different from the outline above.

\subsection{Iterated integral notation}  When all variables are displayed and there is no chance of confusion, we may omit all but one of the integral signs and the subscripts indicating the domain of integration.  So the integral in (\ref{eqn:Blim}) may be written as
\[
\int f(x,y)f(x+s,y+t) f(x+2s,y+2t+2w)g(w) \, dw\, ds \, dt\, dx\, dy.
\]

\section{Eigenvalues and ergodicity of products}\label{sec:Eigenvalues}

An \emph{eigenfunction} of an MPS $\mb X=(X,\mathcal B,\mu,T)$ with \emph{eigenvalue} $\lambda\in \mathbb C$ is an $f\in L^2(\mu)$ satisfying $\|f\|\neq 0$ and $f\circ T=\lambda f$.  Since $\int |f\circ T|\, d\mu = \int |f|\, d\mu$, we have $|\lambda|=1$.  We then have that $|f\circ T|$ is $T$-invariant,  so if $\mb X$ is ergodic, we get that $|f|$ is equal $\mu$-almost everywhere to a constant.  We say an eigenvalue $\lambda$ of $\mb X$ is \emph{nontrivial} if $\lambda\neq 1$.  Note that the eigenfunctions of $\mb X$ are the eigenvectors of the unitary operator $U_T:L^2(\mu)\to L^2(\mu)$ defined by $U_T f = f\circ T$.

Given two MPSs $\mb X = (X,\mathcal B,\mu,T)$ and $\mb Y = (Y,\mathcal D,\nu,S)$, we form the product system $\mb X\times \mb Y=(X\times Y, \mathcal B\otimes \mathcal D,\mu \times \nu, T\times S)$.  For $f\in L^2(\mu)$ and $g\in L^2(\nu)$, we write $f\otimes g$ for the function defined by $f\otimes g(x,y)=f(x)g(y)$.

We need some standard consequences of the following, which is the specialization of Lemma 4.17  of \cite{FurstenbergBook} (p.~91) to the case where $\mathcal H=L^2(\mu)$, $\mathcal H'=L^2(\nu)$ for MPSs $\mb X$ and $\mb Y$ as above, with unitary operators $Uf:=f\circ T$ and $U'g:=g\circ S$.

\begin{lemma}\label{lem:TensorEigenvectors}  Let $\mb X$ and $\mb Y$ be measure preserving systems as above, and let $\mb X\times \mb Y$ be the product system.  Let $h\in L^2(\mu\times \nu)$ be an eigenfunction of $\mb X\times \mb Y$ with eigenvalue $\lambda$, meaning $h\circ (T\times S)=\lambda h$.  Then $h = \sum c_n f_n\otimes g_n$, where $f_n\circ T=\lambda_nf_n$, $g_n\circ S = \lambda_n'g_n$,  $\lambda_n\lambda_n' = \lambda$, and the sequences $\{f_n\}$, $\{g_n\}$ are orthonormal in $L^2(\mu)$ and $L^2(\nu)$ respectively.
\end{lemma}

To deduce Lemma \ref{lem:TensorEigenvectors} from Lemma 4.17 of \cite{FurstenbergBook}, note that if $\mu$ and $\nu$ are measure spaces, $L^2(\mu\times \nu)$ is isomorphic to the tensor product $L^2(\mu)\otimes L^2(\nu)$, and the obvious isomorphism identifies $U_{T\times S}$ with $U_T\otimes U_S$.

The next lemma is a well known consequence of Lemma \ref{lem:TensorEigenvectors}; we omit its proof.

\begin{lemma}\label{lem:ProductErgodic}
	If $\mb X$ and $\mb Y$ are ergodic MPSs, the product system $\mb X\times \mb Y$ is ergodic if and only if $\mb X$ and $\mb Y$ have no nontrivial eigenvalues in common.	
\end{lemma}

Another immediate consequence of Lemma \ref{lem:TensorEigenvectors} is the following.

\begin{lemma}\label{lem:WM}  If $\mb X=(X,\mathcal B,\mu,T)$ and $\mb Y=(Y,\mathcal D,\nu,S)$ are MPSs such that $\mb X\times \mb Y$ is ergodic and $g\in L^2(\nu)$ is orthogonal to every eigenfunction of $\mb Y$, then for every $f\in L^2(\mu)$, $f\otimes g$ is orthogonal to every eigenfunction of $\mb X\times \mb Y$.
\end{lemma}

\section{Eigenfunctions and the Kronecker factor}\label{sec:Kronecker}

Every ergodic MPS $\mb X$ has a factor  $\pi:\mb X\to \mb Z$ where $\mb Z=(Z,\mathcal Z,m,R)$ is a compact abelian group rotation such that every eigenfunction of $\mb X$ is $\pi^{-1}(\mathcal Z)$-measurable.  This factor is called the \emph{Kronecker factor} of $\mb X$, and we write $\int_Z f(s)\, ds$ (or sometimes just $\int f(s)\,ds$) to abbreviate $\int f(s)\, dm_Z(s)$.

The following result is proved in \S 3 of \cite{F77}; we use the notation $L_3$ introduced in \S\ref{sec:MultipleErgodic}.

\begin{theorem}\label{thm:ErgodicRoth}  If $\mb X=(X,\mathcal B,\mu,T)$ is an ergodic MPS with Kronecker factor $\pi:\mb X\to \mb Z$, $\mb Z=(Z,\mathcal Z,m,R)$, and $f_i:X\to [0,1]$, then
	\[\lim_{N\to\infty} \frac{1}{N}\sum_{n=1}^N f_1(T^n x) f_2(T^{2n} x) = \int_Z \tilde{f}_1(\pi(x)+s)\tilde{f}_2(\pi(x)+2s)\, ds,  \qquad (\text{in} \ L^2(\mu))\]
	where $\tilde{f}_i\in L^\infty(m)$ satisfies $\tilde{f}_i\circ \pi=P_{\mb Z}f_i$.  Furthermore
	\begin{equation}\label{eqn:RothIntegralFormula}
		L_3(f,T) = \int_Z\int_Z \tilde{f}(z)\tilde{f}(z+s)\tilde{f}(z+2s)\, dz\, ds \quad \text{for all } f\in L^\infty(\mu).
	\end{equation}
	
\end{theorem}

\subsection{Kronecker factor of a standard 2-step Weyl system}\label{sec:WeylKronecker} 
A \emph{standard $2$-step Weyl system} is an MPS of the form $\mb Y = (Y, \mathcal B, m,S)$, where $Y=\mathbb T^d\times \mathbb T^d$, $d\in \mathbb N$, and $S:Y\to Y$ is defined as $S(x,y)=(x+\alpha,y+x)$, for some fixed $\alpha=(\alpha_1,\dots,\alpha_d)$ generating $\mathbb T^d$.  There is an explicit formula for the orbits of $S$:
\begin{equation}\label{eqn:2stepWeylOrbit}
	S^n(x,y)=(x+n\alpha, y + nx + \tbinom{n}{2}\alpha),
\end{equation}
 which may be verified by induction.  Ergodicity of $\mb Y$ is equivalent to $\mb\alpha$ generating $\mathbb T^d$. For $d=1$ this follows from  Proposition 3.11 of \cite{FurstenbergBook} (p.~67), and the general case follows from a nearly identical proof.  Also explained in \cite{FurstenbergBook} is the Kronecker factor of $\mb Y$: the eigenfunctions of $\mb Y$ are exactly the functions $\chi$ on $Y$ defined by \[
\chi((x_1,\dots,x_d),(y_1,\dots,y_d)):=\exp(2\pi i (n_1x_1+\cdots + n_dx_d))
\] for some $n_j\in\mathbb Z$, so the group of eigenvalues of $\mb Y$ is $\{\exp(2\pi i  (n_1\alpha_1+\cdots +n_d\alpha_d)) : n_j\in \mathbb Z\}$.    
Thus the Kronecker factor of $\mb Y$ is obtained by setting $Z=\mathbb T^d$ and letting  $\pi:\mathbb T^d\times \mathbb T^d \to \mathbb T^d$ be projection onto the first coordinate.  Since the span of the eigenfunctions of $\mb Y$ consists solely of those functions depending on the first coordinate, the orthogonal projection $P_{\mb Z}f(x,y)$ can be written as $(P_{\mb Z}f)(x,y):=\int f(x,y)\, dy$.  Combining this with Theorem \ref{thm:ErgodicRoth}, we have the following.

\begin{observation}\label{obs:KroneckerWeyl}
	The Kronecker factor $(Z,\mathcal Z,m,R)$ of a standard 2-step Weyl system $(\mathbb T^d \times \mathbb T^d,\mathcal D,\mu,S)$ is spanned by functions of the form $f(x,y)=g(x)$ (i.e.~functions depending on only the first coordinate), and for all bounded $f:\mathbb T^d\times \mathbb T^d\to \mathbb C$, we have
	\[
	\lim_{N\to\infty} \frac{1}{N}\sum_{n=1}^N \int f\cdot S^n f\cdot S^{2n}f \, d\mu = \int_{\mathbb T^d}\int_{\mathbb T^d} f'(x)f'(x+s)f'(x+2s)\, dx\, ds,
	\]
where $f':\mathbb T^d\to \mathbb C$ is defined as $f'(x):=\int f(x,y)\, dy$.
\end{observation}

\section{Reduction to Weyl systems}\label{sec:AffineReduction}

The next lemma is one key step in the proof of Lemma \ref{lem:BHKroneckerAlternative}.  Its proof is similar to the proof of Lemma 8.1 in \cite{ABB}.

\begin{lemma}\label{lem:ReductionToWeyl}
Let $\mb X = (X,\mathcal B,\mu,T)$ be a totally ergodic MPS and $f:X\to [0,1]$.  For all $\varepsilon>0$, there is a factor $\pi:\mb X\to \mb Y$ such that
\begin{enumerate}
\item[(i)]  $\mb Y$ is a factor of a standard 2-step Weyl system;

\item[(ii)] setting $\tilde{f}\circ \pi =P_{\mb Y}f$, we have
\[\lim_{N\to\infty} \Bigl|\frac{1}{N}\sum_{n=1}^N  g(n^2 \bm\beta)\int f \cdot T^{n}f\cdot T^{2n}f\, d\mu -  g(n^2\bm\beta) \int \tilde{f}\cdot S^{n}\tilde{f}\cdot S^{2n}\tilde{f}\, d\nu\Bigr|<\varepsilon\]
for every continuous $g:\mathbb T^r\to [0,1]$ and every $\bm\beta\in \mathbb T^r$, for all $r\in \mathbb N$.
\end{enumerate}
If we assume $\bm\beta$ generates $\mathbb T^r$, then (ii) holds for every Riemann integrable $g:\mathbb T^r\to [0,1]$.
\end{lemma}
We prove Lemma \ref{lem:ReductionToWeyl} at the end of this section.  Most of the proof is contained in the next lemma, an application of Theorem B of \cite{FrantzikinakisThreePoly}.  It concerns the \emph{maximal $2$-step affine factor} $\mb A_2$ of an ergodic MPS $\mb X$; see \cite{FrantzikinakisThreePoly} for discussion and exposition. Additionally, we use the standard fact that the Kronecker factor of $\mb X$ is a factor of $\mb A_2$.

If $\mb X$ is an MPS, we write $\mathcal E(\mb X)$ for the group of eigenvalues of $\mb X$ (see \S\ref{sec:Eigenvalues}). We continue to write $e(t)$ for $\exp(2\pi i t)$, and we use the notation $P_{\mb Y}$ introduced in \S\ref{sec:FactorExtension}.

\begin{lemma}\label{lem:ReductionToAffine}
  Let $\mb X=(X,\mathcal B,\mu,T)$ be an ergodic measure preserving system with maximal 2-step affine factor $\mb A_2$ and let $\beta\in [0,1)$.  Then $\mb A_2$ is characteristic for the averages
  \begin{equation}\label{eqn:squarePhaseReduction}
     B_N(f_1,f_2):=\frac{1}{N}\sum_{n=1}^N e(n^2\beta) \cdot T^n f_1 \cdot T^{2n}f_2,
  \end{equation}
  meaning
  \begin{equation}\label{eqn:Characteristic}
  \lim_{N\to\infty} B_N(f_1,f_2) = \lim_{N\to \infty} B_N(P_{\mb A_2}f_1, P_{\mb A_2} f_2)
  \end{equation}
 in $L^2(\mu)$, for all bounded $f_1, f_2$.
 Furthermore, if $\beta$ is irrational and
 \begin{equation}\label{eqn:Spectrum}
 \mathcal E(\mb X) \cap \{e(n\beta)\}_{n\in \mathbb Z}= \{1\}
  \end{equation} then $\lim_{N\to\infty} B_N(f_1,f_2)=0$ in $L^2(\mu)$ for all bounded measurable $f_i$.
\end{lemma}

\begin{remark}
  The existence of $\lim_{N\to\infty} B_N(f_1,f_2)$ is not immediately obvious, but the proof of Lemma \ref{lem:ReductionToAffine} will show that it is a special case of the existence of limits of polynomial multiple ergodic averages found in \cite{FrantzikinakisThreePoly}.
\end{remark}

\begin{proof}  We first dispense with the case where $\beta$ is rational. In this case the sequence $e(n^2\beta)$ is periodic, so we fix a period $p\in \mathbb N$ such that $e((pn+q)^2\beta)=e(q^2\beta)$ for every $n$ and $q\in \mathbb N$. For $0\leq r\leq p$ and $N\in \mathbb N$, we can then write $B_{pN+r}(f_1,f_2)$ as
\[
\frac{1}{pN+r}\Bigl(\sum_{n=pN+1}^{pN+r} e(n^2\beta)\cdot T^nf_1\cdot T^{2n}f_2+\sum_{q=1}^{p}e(q^2\beta)\sum_{n=0}^{N-1}   T^{pn+q}f_1\cdot T^{2(pn+q)}f_2\Bigr).
\]
For large $N$ the sum $\sum_{n=pN+1}^{pN+r} e(n^2\beta)\cdot T^nf_1\cdot T^{2n}f_2$ can be ignored, and we have
\[
\lim_{N\to\infty} B_{N}(f_1,f_2)= \lim_{N\to\infty} \frac{1}{p}\sum_{q=1}^{p}e(q^2\beta) \frac{1}{N}\sum_{n=0}^{N-1}   T^{pn+q}f_1\cdot T^{2(pn+q)}f_2
\]
Theorem A of \cite{FrantzikinakisThreePoly} now implies that the Kronecker factor of $\mb X$ (which is itself a factor of $\mb A_2$) is characteristic for the averages above.  This proves the first assertion of the lemma when $\beta$ is rational.

We now assume $\beta\in (0,1)$ is irrational and consider two cases, based on whether (\ref{eqn:Spectrum}) holds.  When (\ref{eqn:Spectrum}) fails we write the coefficient $e(n^2\beta)$ in terms of $\bar{g}T^{p(n)}g$, where $g\in L^\infty(\mu)$ and $p$ is a polynomial.  When (\ref{eqn:Spectrum}) holds we write $e(n^2\beta)$ as $g_0S^ng_1S^{2n}g_2$ where $\mb Y=(Y,\mathcal D,\nu,S)$ is an ergodic $MPS$ such that $\mb X\times \mb Y$ is ergodic and $g_i\in L^\infty(\nu)$.  In each case we write $B_N(f_1,f_2)$ as a familiar multiple ergodic average and apply known results.

For the first case, we assume (\ref{eqn:Spectrum}) fails. We fix $k\in \mathbb N$ such that $1\neq e(k\beta)\in \mathcal E(\mb X)$, meaning $e(k\beta)$ is a nontrivial eigenvalue of $\mb X$.  Let $g\in L^2(\mu)$ be a corresponding eigenfunction, so that $g\circ T= e(k\beta) g$ and $|g|= 1$ $\mu$-almost everywhere.  Then $e(k\beta)^m = \bar{g}\cdot T^{m}g$ $\mu$-almost everywhere.  In particular,
\begin{equation}\label{eqn:Eigenform}
  e(k\beta)^{kn^2+2nj} = \bar{g}\cdot T^{kn^2+2jn}g \quad \text{ for all } j,n\in \mathbb Z
\end{equation}
in $L^2(\mu)$.  Then
  \begin{align*}
  \lim_{N\to\infty}B_N(f_1,f_2)&= \lim_{N\to\infty}\frac{1}{k}\sum_{j=0}^{k-1} \frac{1}{N}\sum_{n=0}^{N-1} e((nk+j)^2\beta)  T^{nk+j} f_1 \cdot  T^{2nk+2j} f_2\\
   &= \lim_{N\to\infty}   \frac{1}{k}\sum_{j=0}^{k-1} \frac{1}{N}\sum_{n=0}^{N-1}  e(j^2\beta) e(k\beta)^{kn^2+2jn} T^{nk+j} f_1 \cdot  T^{2nk+2j} f_2\\
   &= \lim_{N\to\infty}  \frac{1}{k}\sum_{j=0}^{k-1} e(j^2\beta) \frac{1}{N }\sum_{n=0}^{N-1} \overline{g} \cdot T^{kn^2+2jn}g\cdot  T^{nk+j}f_1 \cdot T^{2nk+2j}f_2.  && \text{by (\ref{eqn:Eigenform})}
  \end{align*}
The polynomial exponents $p_1(n)=kn^2+2jn$, $p_2(n)=nk+j$, $p_3(n)=2nk+2j$ are, in the terminology of \cite{FrantzikinakisThreePoly}, \emph{essentially distinct} and not \emph{type} ($e_1$).  Theorem B of \cite{FrantzikinakisThreePoly} therefore asserts that  $f_1$ and $f_2$ in (\ref{eqn:squarePhaseReduction}) can be replaced with $P_{\mb A_2}f_1$ and $P_{\mb A_2}f_2$, respectively, without changing the value of the limit.  This proves the first assertion of the lemma in the case where $\mathcal E(\mb X)\cap \{e(n\beta)\}_{n\in \mathbb Z}\neq \{1\}$.

Now we assume that (\ref{eqn:Spectrum}) holds. We will prove that $\lim_{N\to\infty} B_N(f_1,f_2)=0$ for all $f_1, f_2\in L^\infty (\mu)$.  This implies (\ref{eqn:Characteristic}), since $P_{\mb A_2} f_i \in L^\infty(\mu)$.

Consider the system $\mb Y=(\mathbb T^2,\mathcal D,m,S)$, where $S(x,y)=(x+\beta,y+x)$; this $\mb Y$ is ergodic since $\beta$ is irrational.   As discussed in \S\ref{sec:WeylKronecker}, the eigenvalues of $\mb Y$ are $\{e(n\beta)\}_{n\in \mathbb Z}$.  Thus $\mb Y$ has no nontrivial eigenvalues in common with $\mb X$, by (\ref{eqn:Spectrum}). The product system $(\mathbb T^2\times X, m \times \mu, S\times T)$ is therefore ergodic, by Lemma \ref{lem:ProductErgodic}.
We will write $B_N(f_1,f_2)$ as an element of $L^2(m\times \mu)$.  First observe that for all $(x,y)\in \mathbb T^2$, we have
\begin{align*}
e(n^2\beta)&= e(y)e(-2y-2nx  -n(n+1) \beta)e(y+2nx+n(2n+1)\beta) \\
&= g_0\cdot g_1(S^n(x,y)) \cdot g_2(S^{2n}(x,y)),
\end{align*} where $g_0(x,y):=e(y)$, $g_1(x,y):=e(-2y)$, $g_2(x,y):=e(y)$.  So
\begin{equation}\label{eqn:Absorb3}
e(n^2\beta) \cdot T^n f_1 \cdot T^{2n}f_2=  g_0\otimes 1_{X} \cdot (S\times T)^n g_1\otimes f_1 \cdot (S\times T)^{2n}g_2\otimes f_2 \in L^2(m\times \mu).
\end{equation}
When computing the limit of the averages of the right hand side in (\ref{eqn:Absorb3}), Theorem \ref{thm:ErgodicRoth} allows us to replace each $g_i\otimes f_i$ with its projection $\phi_i:=P_{\mb Z}(g_i\otimes f_i)$, where $\mb Z$ is the Kronecker factor of $\mb Y\times \mb X$. By Lemma \ref{lem:WM} and Observation \ref{obs:KroneckerWeyl}, $g_i\otimes f_i$ is orthogonal to every eigenfunction of $\mb Y\times \mb X$,  so $\phi_i =0$.  Thus the limit of the averages is $0$ in $L^2(m\times \mu)$.  Since $B_N(f_1,f_2)$ belongs to the natural embedding of $L^2(\mu)$ in $L^2(m\times \mu)$, this proves $\lim_{N\to\infty} B_N(f_1,f_2)=0$ in $L^2(\mu)$.
\end{proof}

\begin{corollary}\label{cor:ReducingToA2} Let $(X,\mathcal B,\mu,T)$ be an ergodic measure preserving system and let $\bm\beta=(\beta_1,\dots,\beta_r)\in \mathbb T^r$. If $g:\mathbb T^r\to \mathbb C$ is continuous and $f_i\in L^\infty(\mu)$, then
\[
\lim_{N\to\infty} \frac{1}{N}\sum_{n=1}^N g(n^2\bm\beta)  \cdot T^{n} f_1 \cdot T^{2n} f_2= \lim_{N\to\infty} \frac{1}{N}\sum_{n=1}^N g(n^2\bm\beta) \cdot T^{n} \bar{f}_1 \cdot T^{2n}\bar{f}_2
\]
in $L^2(\mu)$, where $\mb A_2$ is the maximal 2-step affine factor of $\mb X$, and $\bar{f}_i=P_{\mb A_2}f_i$.
\end{corollary}

\begin{proof}
  Uniformly approximating $g$ by trigonometric polynomials, it suffices to prove the lemma in the case where $g$ is a character of $\mathbb T^r$.  In this case, we can write $g(n^2\bm\beta)$ as $e(n^2\alpha)$ for some $\alpha\in [0,1)$ and apply Lemma \ref{lem:ReductionToAffine}. \end{proof}

\subsection{Nilsystems and their affine factors}

The following is a restatement of Part (i) of Lemma \ref{lem:NilCharacteristic}.

\begin{lemma}\label{lem:BHK46}
	Let $\mb X = (X,\mathcal B, \mu,T)$ be an ergodic MPS, $f_i\in L^\infty(\mu)$, and $\varepsilon>0$.  There is a factor $\pi:\mb X\to \mb Y=(Y,\mathcal D,\nu,S)$ which is a 2-step nilsystem such that for every bounded sequence $(c_n)_{n\in\mathbb N}$  we have
	\begin{equation}\label{eqn:ZeroDensity}
	\limsup_{N\to \infty} \Bigl|\frac{1}{N}\sum_{n=1}^N c_n \int f_0\cdot T^{ n}f_1 \cdot T^{2 n}f_2\, d\mu - c_n\int \tilde{f}_0 \cdot S^{ n}\tilde{f}_1 \cdot  S^{2 n}\tilde{f}_2 \, d\nu\Bigr| <\varepsilon \sup_{n} |c_n|
	\end{equation}
where $\tilde{f}_i\circ \pi:=P_{\mb Y}f_i$.
\end{lemma}

When computing ergodic averages for ergodic 2-step affine nilsystems, the following lemma allows us to specialize to standard Weyl systems.

\begin{lemma}[\cite{FrKrPolyAffine}, Lemma 4.1]\label{lem:FrKrStandard}

Let $T:\mathbb T^d\to \mathbb T^d$ be defined by $T(x)=Ax+b$, where $A$ is a $d\times d$ unipotent integer matrix and $b\in \mathbb T^d$.  Assume furthermore that $T$ is ergodic.  Then $T$ is a factor of an ergodic affine transformation $S:\mathbb T^d\to \mathbb T^d$, where $S=S_1\times S_2\times \cdots \times S_s$ and for $r=1,2,\dots,s$, $S_r: \mathbb T^{d_r} \to \mathbb T^{d_r}$ has the form
\[
S_r(x_{1},,\dots, x_{d_r}) = (x_1+b_r,x_{2}+x_1,\dots, x_{d_r}+x_{d_{r}-1})
\]
for some $b_r\in \mathbb T$.
\end{lemma}

Although not explicitly stated in \cite{FrKrPolyAffine}, the proof there allows us to conclude that we have $d_r\leq D$, where $D$ is the degree of unipotency of $A$.  Furthermore, if $(A-I)^2=0$, as is the case when $T$ is 2-step affine, then we can take $d_r\leq 2$ for each $r$. For convenience, we may also assume that $d_r=2$ for each $r$, and therefore $s=1$.  With these specializations, the system given by $S$ above is a standard 2-step Weyl system.

\begin{proof}[Proof of Lemma \ref{lem:ReductionToWeyl}]
	Fix a totally ergodic MPS $\mb X = (X,\mathcal B,\mu,T)$, bounded measurable functions $f_i$ on $X$, $r\in \mathbb N$, and $\bm\beta\in \mathbb T^r$. Let $g:\mathbb T^r\to [0,1]$ be continuous, and let $\varepsilon>0$. Consider the averages
	\[
	A_N: = \frac{1}{N}\sum_{n=1}^N g(n^2\bm\beta)\int f_0 \cdot T^{n}f_1\cdot T^{2n}f_2\, d\mu.
	\]
	First apply Lemma \ref{lem:BHK46} to find a 2-step nilsystem $\mb Y_0=(Y_0,\mathcal D_0,\nu_0,S)$ satisfying (\ref{eqn:ZeroDensity}) with $c_n=g(n^2\bm\beta)$, and write $B_N$ for the averages
	\[\frac{1}{N}\sum_{n=1}^N g(n^2\bm\beta) \int \tilde{f}_0\cdot S^{n}\tilde{f}_1 \cdot S^{2n}\tilde{f}_2\, d\nu.\]
	Our application of Lemma \ref{lem:BHK46} means that $\limsup_{N\to \infty} |A_N-B_N|<\varepsilon$.
	
	By Corollary \ref{cor:ReducingToA2}, the factor $\mb Y:=\mb A_2(\mb Y_0)$ is characteristic for the averages $B_N$: we may replace each $\tilde{f}_i$ with $P_{\mb Y}\tilde{f}_i$ without affecting $\lim_{N\to\infty} B_N$.  The total ergodicity of $\mb X$ implies every factor of $\mb X$ is also totally ergodic; in particular $\mb A_2(\mb Y_0)$ is totally ergodic.  By Lemma \ref{lem:2AofY}, we conclude that $\mb A_2(\mb Y_0)$ is isomorphic to a unipotent $2$-step affine transformation on a finite-dimensional torus, and Lemma \ref{lem:FrKrStandard} allows us to conclude that $\mb A_2(\mb Y)$ is a factor of a standard 2-step Weyl system.

To obtain the remark after (ii) in the statement of the lemma, apply Lemma \ref{lem:RiemannCoefficients} with $y_n = n^2\bm\beta$ and $v_n = \int f \cdot T^{n}f\cdot T^{2n}f\, d\mu -  \int \tilde{f}\cdot S^{n}\tilde{f}\cdot S^{2n}\tilde{f}\, d\nu$.  We may apply Lemma \ref{lem:RiemannCoefficients} since the Weyl criterion implies $n^2\bm \beta$ is uniformly distributed in $\mathbb T^r$ whenever $\bm\beta$ is generating. \end{proof}

\begin{remark}
Our proof of Lemma \ref{lem:ReductionToWeyl} needs the hypothesis of total ergodicity to conclude that $\mb A_2(\mb Y)$ is isomorphic to a $2$-step affine transformation on a finite-dimensional torus. Without this hypothesis, $\mb A_2(\mb X)$ may be more complicated: the underlying space may be disconnected, and may even have uncountably many connected components.  In particular, the Kronecker factor of $\mb X$ could be isomorphic to a rotation on a compact uncountable totally disconnected abelian group (such as the profinite compactification of $\mathbb Z$).  This would cause two problems in the sequel: first, in Section \ref{sec:AffineLimits} we exploit the fact that the connected component of a closed subgroup $\Lambda$ of $\mathbb T^d$ has finite index in $\Lambda$ (although this may not be crucial).  Second, we simply lack a convenient algebraic description of affine systems defined on disconnected groups, and such a description is required for our computation in Proposition \ref{prop:WeakLimit}.

For similar reasons, we cannot prove Lemma \ref{lem:ReductionToWeyl} starting with an arbitrary totally ergodic $\mb X$ and passing immediately to $\mb A_2(\mb X)$.   While disconnectedness will not be a problem, it is possible that the Kronecker factor of $\mb X$ is a group rotation on an infinite dimensional torus, or a solenoid, and then $\mb A_2(\mb X)$ could be an affine transformation on such a group, which does not fit the hypothesis of Lemma \ref{lem:FrKrStandard}.
\end{remark}

\section{Joinings of groups}\label{sec:Joinings}

Given two compact abelian groups $Z$ and $W$ with cartesian product $Z\times W$, write $\pi_1$ and $\pi_2$ for the projection maps onto $Z$ and $W$, respectively.  We say a subgroup $\Gamma\subseteq Z\times W$ is a \emph{joining of $Z$ with $W$} if $\Gamma$ is closed and $\pi_1:\Gamma\to Z$ and $\pi_2:\Gamma \to W$ are both surjective.

\begin{observation}\label{obs:Joining} If $\alpha\in Z$ and $\beta\in W$ are generating elements, then the closed subgroup $\Gamma$ of $Z\times W$ generated by $(\alpha,\beta)$ is a joining of $Z$ with $W$: $\pi_1(\Gamma)$ is generated by $\alpha$ and $\pi_2(\Gamma)$ is generated by $\beta$.
\end{observation}

Joinings arise naturally in the computation of multiple ergodic averages.  For example, let $\Gamma:=\{(t,t)\}: t\in Z\}$ ($=$ the diagonal of $Z\times Z$), so that $\Gamma$ is a joining of $Z$ with itself.  Then we can write the integral on the right-hand side of (\ref{eqn:RothIntegralFormula}) as
\begin{equation}\label{eqn:DiagonalExample}
\int_{\Gamma} \int_Z  f(x)f(x+\pi_1(t))f(x+2\pi_2(t)) \, dx \, dm_{\Gamma}(t).
\end{equation}
The notation $\pi_i(t)$ will be cumbersome in our formulas, so we adopt the following abbreviation.
\begin{notation}\label{not:Indices}  If $\Gamma$ is a joining of $Z$ with $W$ and $t\in \Gamma$, we write $t_1$ for $\pi_1(t)$ and $t_2$ for $\pi_2(t)$.
\end{notation}
So the integral in (\ref{eqn:DiagonalExample}) can be written as $\int_\Gamma \int_Z f(x)f(x+t_1)f(x+2t_2)\, dx \, dm_{\Gamma}(t)$.

The joinings we consider will be closed subgroups of $\mathbb T^d\times\mathbb T^r$; this allows us to exploit the well known structure of such groups (detailed in \cite{RudinGroupsBook}, for example).

\begin{observation}\label{obs:ComponentOfJoining} If $\Gamma$ is a joining of $\mathbb T^d$ with $\mathbb T^r$, then its identity component is also a joining of these groups.  To see this, note that since $\Gamma$ is a closed subgroup of a finite-dimensional torus, its identity component $\Gamma_0$ has finite index in $\Gamma$.  The images of $\pi_1$ and $\pi_2$ therefore have finite index in their respective codomains $\mathbb T^d$ and $\mathbb T^r$. Since these codomains are connected, they have no proper  closed finite index subgroups, so the images $\pi_1(\Gamma_0)$, $\pi_2(\Gamma_0)$ must equal their respective codomains.
\end{observation}

If $G$ is a compact abelian group and $H$ is a closed subgroup, $m_{H}$ denotes Haar probability measure on $H$.  If $H'$ is a coset $H+t$ of $H$,  $m_{H'}$ denotes Haar measure on $H'$, i.e.~the measure given by $\int f\, dm_{H'}:= \int f(x+t)\, dm_H(x)$.

\begin{definition}\label{def:AffineJoining}
  If $\Gamma_0$ is a joining of $Z$ with $W$, $\Gamma_j, j\leq k$ is a collection of cosets of $\Gamma_0$, and $c_j \in [0,1]$ satisfy $\sum_{j} c_j=1$, we say that the $\Gamma_j$ and $c_j$ form an \emph{affine joining} $\Gamma$ of $Z$ with $W$, and define integration over $\Gamma$ by
  \[
  \int f \, dm_{\Gamma} := \sum_{j} c_j \int f\, dm_{\Gamma_j}.
  \]
\end{definition}
\noindent For example, $\Gamma_0 = \{(x,2x):x\in \mathbb T\}$, $\Gamma_1 = \{(x+\frac{1}{4},2x):x\in \mathbb T\} \subseteq \mathbb T\times \mathbb T$, $c_0 = \frac{1}{3}$, and $c_1 = \frac{2}{3}$ determine an affine joining $\Gamma$ of $\mathbb T$ with $\mathbb T$, and \[
\int f\, dm_{\Gamma} = \frac{1}{3}\int f(x,2x) \, dx + \frac{2}{3} \int f(x+\tfrac{1}{4},2x)\, dx.\]

  \section{Application of Kronecker's and Weyl's theorems}\label{sec:Weyl}

The limits of ergodic averages we consider will be computed as integrals over affine joinings.  To compute them explicitly, we need the following well known results of Kronecker and Weyl.

Given a compact abelian group $Z$ and $\alpha_1,\dots,\alpha_d\in Z$, we write $\langle \alpha_1,\dots, \alpha_d\rangle$ for the subgroup of $Z$ generated by these elements.  We write $\overline{\langle \alpha_1,\dots, \alpha_d\rangle}$ for its closure.

\begin{lemma}[Kronecker's criterion]\label{lem:KroneckerCrit} Let $\alpha_1,\dots,\alpha_d$ be elements of a compact abelian group $Z$.  Then $\overline{\langle \alpha_1,\dots, \alpha_d\rangle}=Z$ if and only if for every nontrivial character $\chi\in \widehat{Z}$, $\chi(\alpha_j)\neq 1$ for at least one of the $\alpha_j$.
\end{lemma}
Weyl's theorem on uniform distribution of polynomials (\cite{Weyl1916}, or Theorem 3.2 of \cite{KuipersNiederreiter}) says that if $p(x)=c_mx^m + c_{m-1}x^{m-1} + \cdots + c_0$ is a polynomial with real coefficients and at least one of the $c_j$ with $j>0$ is irrational, then
\[
\lim_{N\to\infty} \frac{1}{N}\sum_{n=1}^N e(p(n)) = 0.
\]
As usual $e(t)$ denotes $\exp(2\pi i t)$.

\begin{lemma}\label{lem:Weyl}
 Let $Z$ be a compact abelian group, let $\alpha$, $\beta\in Z$, and let $\chi\in \widehat{Z}$ be such that $\chi(\alpha)$, $\chi(\beta)$ are not both roots of unity. Then $\lim_{N\to\infty} \frac{1}{N}\sum_{n=1}^N \chi(n\alpha + n^2\beta)=0$.
\end{lemma}
\begin{proof}
  Write $\chi(n\alpha + n^2\beta)$ as $\chi(\alpha)^n\chi(\beta)^{n^2} = e(n\gamma_1 + n^2\gamma_2)$, where at least one of $\gamma_1, \gamma_2\in [0,1)$ is irrational.  Weyl's theorem then implies the limit of the averages is $0$.
\end{proof}

  \begin{lemma}\label{lem:ConnectedWeyl} Let $Z$ be a compact abelian group with Haar probability measure $m$ and let $\alpha, \beta$ generate $Z$.

  \begin{enumerate}
  	
  \item[(i)] If $Z$ is connected, then for all continuous $f:Z\to \mathbb C$, we have
  \begin{equation}\label{eqn:PolyAverage}
  \lim_{N\to\infty} \frac{1}{N}\sum_{n=1}^N f(n\alpha + n^2\beta) = \int f\, dm.
  \end{equation}
  \item[(ii)]   If $Z$ has finitely many connected components $Z_j$, then the limit above can be written as $\sum c_j \int f\, dm_{Z_j}$ for some nonnegative $c_j$ with $\sum c_j=1$.
  \item[(iii)]  For fixed $\beta$, if $\alpha, \alpha'\in Z$ are such that $\overline{\langle \alpha \rangle} = \overline{\langle \alpha'\rangle}$ and $\overline{\langle \alpha \rangle}$ is connected, we have
      \begin{equation}\label{eqn:alphaalphaprime}
      \lim_{N\to\infty} \frac{1}{N}\sum_{n=1}^N f(n\alpha + n^2\beta) = \lim_{N\to\infty} \frac{1}{N}\sum_{n=1}^N f(n\alpha' + n^2\beta).
      \end{equation}

  \end{enumerate}
  \end{lemma}

  \begin{proof}
(i) Approximating $f$ by trigonometric polynomials, it suffices to prove the special case where $f$ is a nontrivial character $\chi\in \widehat{Z}$. Under this assumption we will show that the limit of the averages in (\ref{eqn:PolyAverage}) is $0$.  In this case $f(n\alpha+n^2\beta)=\chi(\alpha)^n\chi(\beta)^{n^2}$.  Connectedness of $Z$ implies $\chi^n\not\equiv 1$ for all $n\in \mathbb N$.  Since $\alpha$ and $\beta$ generate $Z$, Lemma \ref{lem:KroneckerCrit} implies $\chi(\alpha)^n\neq 1$ or $\chi(\beta)^n\neq 1$ for all $n\in \mathbb N$.  Lemma \ref{lem:Weyl} then implies $\lim_{N\to\infty} \frac{1}{N}\sum_{n=1}^N \chi(n\alpha + n^2\beta)=0$.

\smallskip

\noindent (ii)  Assuming $Z$ has finitely many connected components $Z_j$ and identity component $Z_0$, let $A_j:=\{n\in \mathbb Z:n\alpha+n^2\beta\in Z_j\}$, and let $p$ be the index of $Z_0$ in $Z$.  We claim that each $A_j$ is a union of infinite arithmetic progressions of the form $p\mathbb Z+q$.  To prove this it suffices to prove $A_j+p = A_j$.  To do so, observe that $p\alpha, p\beta \in Z_0$.  We will show that if $n\in A_j$, then $n-p\in A_j$; in other words, if $n\alpha + n^2\beta\in Z_j$, then $(n+p)\alpha + (n+p)^2\beta \in Z_j$.  Now fix $n,j$ with $n\alpha+n^2\beta\in Z_j$.  Then
\[(n+p)\alpha + (n+p)^2\beta = n\alpha +n^2\beta + p\alpha + (2n+p)p\beta \in Z_j+Z_0=Z_j,\] as desired.  Similarly, we can show that if $n\in A_j$, then $n+p\in A_j$, so that $A_j+p=A_j$.

Fix $q\in \mathbb Z$.  We claim $\alpha_0:=(1+2q)p\alpha$ and $\beta_0:=p^2\beta$ generate $Z_0$.  To see this, note that the closed subgroup they generate is contained in $Z_0$, and has finite index in the subgroup generated by $\alpha$ and $\beta$, while  $Z_0$ has no proper finite index closed subgroup.

We decompose the limit in (\ref{eqn:PolyAverage}) as
\begin{equation}\label{eqn:pqRewrite}
\frac{1}{p}\sum_{q=0}^{p-1} \lim_{N\to\infty} \frac{1}{N}\sum_{n=1}^N f((pn+q)\alpha + (pn+q)^2\beta).
\end{equation}
Thinking of $q$ as fixed, so that $(pn+q)\alpha + (pn+q)^2\beta\in Z_i$ for some $i$, it suffices to prove that the limit is $0$ when  $f$ is a character $\chi$ of $Z$ which is not constant on $Z_i$ (and therefore not constant on $Z_0$). We fix such a $\chi$ and write
\begin{align*}
\chi((pn+q)\alpha + (pn+q)^2\beta) &= \chi(q\alpha+q^2\beta)\chi(n(1+2q)p\alpha+n^2p^2\beta)\\
&=\chi(q\alpha+q^2\beta)\chi(n\alpha_0+n^2\beta_0).
\end{align*} Since $\alpha_0$ and $\beta_0$ generate $Z_0$, we have
\[	
\lim_{N\to\infty} \frac{1}{N}\sum_{n=1}^N \chi((pn+q)\alpha + (pn+q)^2\beta) = \chi(q\alpha+q^2\beta) \lim_{N\to\infty}  \frac{1}{N}\sum_{n=1}^N \chi(n\alpha_0+n^2\beta_0)	= 0,
\]
by part (i). This shows that the averages in (\ref{eqn:pqRewrite}) converge to $0$ when $f$ is a character which is not constant on $Z_i$, completing the proof of (ii).

To prove (iii), fix $\alpha,\alpha', \beta\in Z$ and assume $H:=\overline{\langle \alpha\rangle}=\overline{\langle \alpha'\rangle}$ is a connected subgroup of $Z$. It suffices to prove that (\ref{eqn:alphaalphaprime}) holds when $f$ is a character $\chi$ of $Z$.  If $\chi(H)=\{1\}$, then $\chi(n\alpha+n^2\beta)=\chi(n\alpha'+n^2\beta)=\chi(n^2\beta$), so the averages in (\ref{eqn:alphaalphaprime}) are equal.  Now assume $\chi(H)\neq \{1\}$.  We will prove that both sides of (\ref{eqn:alphaalphaprime}) are $0$. First note that $\chi(H)=\{z\in \mathbb C:|z|=1\}$, since $H$ is compact and connected, and its image under $\chi$ is a nontrivial compact connected subgroup of $\mathcal S^1$.  Since $\alpha$ and $\alpha'$ generate dense subgroups of $H$, $\chi(\alpha)$ and $\chi(\alpha')$ generate dense subgroups of $\chi(H)$, hence they are both not roots of unity.  Lemma \ref{lem:Weyl} then implies both limits in (iii) are $0$. \end{proof}

\begin{remark}
  Part (iii) of Lemma \ref{lem:ConnectedWeyl} says that when $\beta$ is fixed and $\overline{\langle \alpha \rangle} = \overline{\langle \alpha'\rangle}$ is connected, the $c_j$ provided by part (ii) do not change when $\alpha'$ replaces $\alpha$.
\end{remark}

\section{The Roth integral and Fourier coefficients}\label{sec:RothIntegral}
Let $Z$ be a compact abelian group with Haar probability measure $m$ and $f:Z\to [0,1]$.  We examine the multilinear form which ``counts 3-term arithmetic progressions'' in the support of $f$:
\[
I_3(f):=\int f(z)f(z+t)f(z+2t)\, dm(z)\, dm(t)
\]
Roth \cite{RothFrench,RothEnglish} (cf.~\cite{GowersSz}) and Furstenberg \cite{F77} observed that if $|\hat{f}(\chi)|$ is small for all nontrivial $\chi\in \widehat{Z}$  then $I_3(f)\approx \bigl(\int f\, dm\bigr)^3$.  Lemma \ref{lem:SmallFourierToW} is a minor generalization of this fact; to state it we first introduce some notation.

Let $W=Z/K$ be a quotient of $Z$ by a closed subgroup $K$.  For $f\in L^2(m)$, let
\begin{equation}\label{eqn:PrimeDef}
  f'(z):= \int_K f(z+y) \, dm_K(y).
\end{equation}
Let $\pi : Z\to W$ be the quotient map, and identify $\widehat{W}$ with  $\{\chi\circ \pi: \chi \in \widehat{W}\}\subseteq \widehat{Z}$. We have
\begin{equation}	\label{eqn:f'chi1}\widehat{f'}(\chi)=\begin{cases}
                	                                       \hat{f}(\chi), & \mbox{if } \chi\in\widehat{W} \\
                	                                       0, & \mbox{if } \chi \notin\widehat{W}.
                	                                     \end{cases}
\end{equation}
To see this, note that for $\chi\in \widehat{W}$, we have $\chi(z+y)=\chi(z)$ for all $y\in K$, so

\begin{align*}
  \hat{f}(\chi) = \int f(z)\overline{\chi(z)}\, dm(z)  &= \int \int f(z+y)\overline{\chi(z+y)} \, dm(z) \, dm_K(y)\\
  &= \int \int f(z+y) \, dm_K(y) \overline{\chi(z)}\, dm(z)\\
  &= \int f' \overline{\chi}\, dm\\
  &=\widehat{f'}(\chi).
\end{align*}
Now for $\chi\notin \widehat{W}$, there exists $t\in K$ such that $\chi(t)\neq 1$. Since $f'(z+s)=f'(z)$ for all $s\in K$, we have
\begin{align*}
 \widehat{f'}(\chi) =	\int f'(z) \overline{\chi(z)} \, dm(z) &= \int f'(z+t)\overline{\chi(z+t)} \, dm(z)\\
	&= \int f'(z)\overline{\chi(z+t)}\, dm(z)\\
	&= \overline{\chi(t)}\int f'(z)\overline{\chi(z)}\, dm(z)\\
	&= \overline{\chi(t)}\widehat{f'}(\chi).
\end{align*}
So $\widehat{f'}(\chi) = \overline{\chi(t)}\widehat{f'}(\chi)$, which is possible only if $\widehat{f'}(\chi)=0$.

Below we will use $dz$ and $dt$ to indicate integrate over all of $Z$ with respect to the displayed variable.  Integration over $K$ will be indicated by $\,dm_K$.
\begin{lemma}\label{lem:SmallFourierToW}
With $Z$, $K$, and $W$ as defined above, let $f_0, f_1, f_2 \in L^\infty(m)$, and write
\begin{align*}
	I &:= \int \int f_0(z)f_1(z+t)f_2(z+2t)\, dz \, dt\\
	I_W&:= \int \int f_0'(z)f_1'(z+t)f_2'(z+2t) \, dz \, dt.
\end{align*}
 Suppose $|\hat{f}_2(\chi)|\leq \kappa$ for all $\chi \in \widehat{Z} \setminus \widehat{W}$.  Assuming the map $\chi\mapsto \chi^2$ is injective on $\widehat{Z}$, we have
\begin{equation}\label{eqn:IIW}
 |I-I_W| \leq \kappa \|f_0\|_{L^2(m)}\|f_1\|_{L^2(m)}.
\end{equation}
\end{lemma}

\begin{proof}
	Let $I_2= \int \int f_0(z)f_1(z+t)f_2'(z+2t)\, dz\, dt$.  We will prove that  \begin{equation}
		\label{eqn:II2}|I-I_2|\leq \kappa\|f_0\|_{L^2(m)}\|f_1\|_{L^2(m)}
	\end{equation}
and that $I_2 = I_W$.  We first prove the special case where each $f_i$ is a trigonometric polynomial. Expanding each $f_i$ as $\sum_{\chi \in \widehat{Z}} \hat{f}_i(\chi)\chi$ and simplifying, we get
\begin{align*}
	I &= \sum_{\chi, \psi,\tau \in \widehat{Z}} \int \hat{f}_0(\chi)\chi(z) \hat{f}_1(\psi)\psi(z+t) \hat{f}_2(\tau)\tau(z+2t)\, dz \, dt\\
	&= \sum_{\chi, \psi, \tau \in \widehat{Z}}  \int \hat{f}_0(\chi) \hat{f}_1(\psi) \hat{f}_2(\tau)  \chi\psi\tau(z) \psi\tau^2(t)\, dz\, dt\\
	&= \sum_{\chi,\psi,\tau\in\widehat{Z}} \hat{f}_0(\chi) \hat{f}_1(\psi) \hat{f}_2(\tau) \int \chi\psi\tau(z)\, dz \int \psi \tau^2(t)\, dt.
\end{align*}
At least one of  $\int \psi \tau^2(t)\, dt$ or $\int \chi\psi\tau(z)\, dz$ is zero unless $\psi\tau^2$ and $\chi\psi\tau$ are both trivial; this triviality occurs exactly when $\psi = \tau^{-2}$ and $\chi=\tau$.  The sum in the last line above may therefore be restricted to values of $\chi,\psi,$ and $\tau$ satisfying these identities, and we get
\begin{align*}
I=\sum_{\tau \in \widehat{Z}} \hat{f}_0(\tau)\hat{f}_1(\tau^{-2})\hat{f}_2(\tau).
\end{align*}
As noted in (\ref{eqn:f'chi1}), $\hat{f}_2'(\tau) = \hat{f}_2(\tau)$ for $\tau \in \widehat{W}$ and $\widehat{f'}_2(\tau)=0$ for $\tau \notin \widehat{Z}\setminus \widehat{W},$ so
\[
I_2 = \sum_{\tau \in \widehat{W}} \hat{f}_0(\tau)\hat{f}_1(\tau^{-2})\hat{f}_2(\tau).
\]	
Then
\begin{align*}
	|I - I_2| & = \Bigl|\sum_{\tau \notin \widehat{W}} \hat{f}_0(\tau)\hat{f}_1(\tau^{-2})\hat{f}_2(\tau)\Bigr|\\
	&\leq \sum_{\tau\in \widehat{Z}} \kappa |\hat{f}_0(\tau) \hat{f}_1(\tau^{-2  })|  && \text{since } |\hat{f}_2(\tau)|<\kappa \text{ for } \tau\notin \widehat{W}\\
	&\leq \kappa \|\hat{f}_0\|_{l^2}\|\hat{f}_1\|_{l^2}  && \text{Cauchy-Schwarz, assuming } \tau\mapsto \tau^2 \text{ is injective}\\
	&= \kappa \|f_0\|_{L^2(m)}\|f_1\|_{L^2(m)}, && \text{Plancherel}
\end{align*}
where $\|\cdot\|_{l^2}$ denotes the $l^2$ norm for functions on $\widehat{Z}$.

To prove $I_2=I_W$, replace $t$ with $t+s$ in the $dt$ integral in $I_2$, then integrate $s$ over $K$, using the fact that $f_2'(z+s)=f_2(z)$ for all $z\in Z$, $s\in K$:
\begin{align*}
	I_2 &= \int \int f_0(z) f_1(z+t)f_2'(z+2t) \, dz \, dt\\
	&= \int_K \int \int f_0(z) f_1(z+t+s) f_2'(z+2t+2s) \, dz \, dt\, dm_K(s)\\
	&= \int \int f_0(z) \int_K f_1(z+t+s)\, dm_K(s)\,  f_2'(z+2t) \, dz \, dt \\
	&= \int \int f_0(z) f_1'(z+t) f_2'(z+2t) \, dz\, dt.
\end{align*}
A similar manipulation, replacing $z$ with $z+s$, lets us replace $f_0$ with $f_0'$, completing the proof that $I_2=I_W$, and hence $|I - I_W|\leq \kappa \|f_0\|_{L^2(m)}\|f_1\|_{L^2(m)}$.  This proves (\ref{eqn:II2}).
\end{proof}

\section{Annihilating characters}\label{sec:Annihilating}

Let $d,r\in\mathbb N$, let $f:\mathbb T^d\times \mathbb T^d\to \mathbb C$, $g:\mathbb T^r\to \mathbb C$, and let $\Gamma$ be an affine joining of $\mathbb T^d$ with $\mathbb T^r$ (Definition \ref{def:AffineJoining}).  The limits we compute in the proof of Lemma \ref{lem:BHKroneckerAlternative} will contain functions of the form $f*_{\Gamma} g:\mathbb T^d\times \mathbb T^d\to \mathbb C$, defined by
\begin{equation}\label{def:FstarG}
f*_\Gamma g(x,y):=\int f(x,y+2\pi_1(w))g(\pi_2(w))\, dm_{\Gamma}(w).
\end{equation}  The next two lemmas let us bound the Fourier coefficients of $f*_{\Gamma} g$.  We use the abbreviation $w_i$ for $\pi_i(w)$ introduced in Notation \ref{not:Indices}.

\begin{lemma}\label{lem:BHannihilates}
	Let $k<r\in \mathbb N$ and $U\subseteq \mathbb T^r$ be an approximate Hamming ball of radius $(k,\eta)$, $\eta>0$.  Then
	
	\begin{enumerate}
	\item[(i)] If $Z$ is a compact abelian group, $\pi:Z\to \mathbb T^r $ is a continuous homomorphism, and $\chi_1,\dots,\chi_k\in \widehat{Z}$ are nontrivial then there is a cylinder function $g$ subordinate to $U$ such that $\widehat{g_{s}\circ \pi}(\chi_j)=0$ for each $j$ and each translate $g_{s}$ of $g$.
	
	\item[(ii)] If $\Gamma$ is an affine joining of $\mathbb T^d$ with $\mathbb T^r$ and $\chi_1,\dots,\chi_k\in \widehat{\mathbb T}^d$ are nontrivial, then there is a cylinder function $g$ subordinate to $U$ such that $\int \chi_j(w_1) g(w_2)\, dm_{\Gamma}(w) = 0$ for each $j\leq k$.

	\end{enumerate}
\end{lemma}

\begin{proof} (i)	Let $\chi_j \in \widehat{Z}$ for $j\leq k$, and let $K$ be the kernel of $\pi$.  We first consider those $\chi_j$ where $\chi_j|_K$ is constant.  In this case, $\chi_j$ can be written as $\chi_j'\circ \pi$, where $\chi_j'\in \widehat{\mathbb T}^r$, and $\int (g_{s}\circ \pi)\, \overline{\chi}_j\, dm$ can be written as
\[
\int  g_{s}\, \overline{\chi}_j'\, dm_{\mathbb T^r}.
\] So choose $g$ by Lemma \ref{lem:VAnnihilates} to make these integrals vanish for such $\chi_j$.  For those $j$ where $\chi_j|_K$ is not constant, write $\int_Z f(z)\, dz$ as $\int_{Z} \int_{K} f(z+t) dm_K(t)\, dm(z).$  Then
	\begin{align*}
	\widehat{g_s\circ \pi}(\chi_j)=\int  g_{s}(\pi(z)) \overline{\chi_j(z)}\, dz &= \int \int g_{s}(\pi(z+t)) \overline{\chi_j(z+t) }\, dm_K(t)\, dz \\
	&= \int  g_{s}(\pi(z))\overline{\chi_j(z)} \, dz\, \int_{K} \overline{\chi_j(t)}\, dm_K(t) \\
	&=0,
	\end{align*}
	where the last line follows from the fact that $\chi_j|_K$ is a nontrivial character of $K$.
	
\medskip

\noindent (ii)  Since $\Gamma$ is an affine joining of $\mathbb T^d$ with $\mathbb T^r$, by definition, there is a joining $\Gamma_0$ so that the integral over $\Gamma$ is a convex combination of integrals over translates of $\Gamma_0$.  To prove (ii) it therefore suffices to find a $g$ subordinate to $U$ so that $\int \chi(w_1)g(w_2)\, dm_{\Gamma+t}(w)=0$ for every $t\in \mathbb T^{d}\times \mathbb T^{r}$.  We will use the identity
\begin{equation}\label{eqn:TranslateGammaStar}
\int \chi(w_1)g(w_2)\, dm_{\Gamma+t}(w) = \chi(\pi_1(t)) \int_{\Gamma_0} \chi(w_1)g_{-\pi_2(t)}(w_2)\, dm_{\Gamma_0}(w),
	\end{equation}which follows from the manipulations
\begin{align*}
\int_{\Gamma_0+t} \chi(w_1)g(w_2)\,  dm_{\Gamma_0+t}(w) &= \int_{\Gamma_0} \chi(\pi_1(w+t))g(\pi_2(w+t)) \, dm_{\Gamma_0}(w)\\
&= \int_{\Gamma_0} \chi(\pi_1(t)) \chi(\pi_1(w))g(\pi_2(w)+\pi_2(t))\, dm_{\Gamma_0}(w)\\
&= \chi(\pi_1(t)) \int_{\Gamma_0} \chi(w_1)g_{-\pi_2(t)}(w_2)\, dm_{\Gamma_0}(w).
\end{align*}
 We can consider the functions $z\mapsto \chi_j(\pi_1(z))$ as characters $\tilde{\chi}_j$ $\Gamma_0$.  These characters are nontrivial since $\pi_1:\Gamma_0\to \mathbb T^d$ is surjective, so we can apply Part (i) of the present lemma (with $\Gamma_0$ in place of $Z$ and $\pi_2$ in place of $\pi$) to find $g$ subordinate to $U$ so that $\int_{\Gamma_0} \chi_j(\pi_1(w)) g_{s}(\pi_2(w))\, dm_{\Gamma_0}(w)=: \widehat{g_s\circ \pi_2}(\tilde{\chi}_j)=0$ for every translate $g_{s}$ of $g$ and every $j\leq k$.  In light of (\ref{eqn:TranslateGammaStar}), this proves Part (ii).
\end{proof}

The expression $f*_{\Gamma} g$ in the next lemma is defined in (\ref{def:FstarG}); for $h:\mathbb T^d\times\mathbb T^d\to \mathbb C$, $h':\mathbb T^d\to \mathbb C$ is defined as $h'(x):=\int_{\mathbb T^d} h(x,y) \, dy$.

\begin{lemma}\label{lem:BHconvolvesToUniform}
	Let $k,d,r\in \mathbb N$, and let $\Gamma$ be an affine joining of $\mathbb T^d$ with $\mathbb T^r$. Let $U\subseteq \mathbb T^r$ be an approximate Hamming ball of radius $(k,\eta)$ for some $\eta>0$.
	
	Let $f:\mathbb T^d\times\mathbb T^d\to [0,1]$.  Then there is a cylinder function $g$ subordinate to $U$ such that
	\begin{align}\label{eqn:chi123bound}
		|\widehat{f*_{\Gamma} g}(\chi,\psi)|&< k^{-1/2} \qquad \text{whenever $\psi$ is nontrivial},	\\
\label{eqn:SamePrime}
    	(f*_{\Gamma} g)' &= f'.
   	\end{align}
\end{lemma}

\begin{proof}
	Let $(\chi,\psi)_j$, $j\leq k$, denote the characters of $\mathbb T^d\times \mathbb T^d$ having the $k$ largest values of $|\hat{f}(\chi,\psi)|$, among those where  $\psi$  is nontrivial.  Lemma \ref{lem:NumberOfLargeCoefficients} implies
	\begin{equation}\label{eqn:OtherChars}
		|\hat{f}(\chi,\psi)|<k^{-1/2} \text{ for all } (\chi,\psi)\notin \{(\chi,\psi)_1,\dots, (\chi,\psi)_k\} \text{ with } \psi \text{ nontrivial.}
	\end{equation}
Applying Part (ii) of Lemma \ref{lem:BHannihilates} with $\psi^2$ in place of the $\chi_j$, we may choose a cylinder function $g$ subordinate to $U$  such that
\begin{equation}\label{eqn:GammaPsi0}
\int_{\Gamma} \psi(2w_1)g(w_2)\, dm_{\Gamma}(w)=0 \text{ for every } \psi \text{ appearing in the } (\chi,\psi)_j \text{ selected above.}
\end{equation}
	Expand $f$ as a Fourier series $\sum_{(\chi,\psi)\in \widehat{\mathbb T}^d\times \widehat{\mathbb T}^d} \hat{f}(\chi,\psi)\chi\psi$,
	so that
	\begin{align*}
		f*_{\Gamma}g&=\int f(x,y+2w_1)g(w_2)\, dm_{\Gamma}(w) \\
		&= \sum_{(\chi,\psi)} \hat{f}(\chi,\psi)\chi(x)\int \psi(y+2w_1)g(w_2)\, dm_{\Gamma}(w)\\
		&= \sum_{(\chi,\psi)} \hat{f}(\chi,\psi)\chi(x)\psi(y)\int \psi(2w_1)g(w_2)\, dm_{\Gamma}(w),
	\end{align*}
and
\begin{equation}\label{eqn:fStargHat}
\widehat{f*_{\Gamma} g}(\chi,\psi) = \hat{f}(\chi,\psi) \int \psi(2w_1)g(w_2)\, dm_{\Gamma}(w).
\end{equation}
 The integral in (\ref{eqn:fStargHat}) is $0$ when $(\chi,\psi)$ are among the $(\chi,\psi)_j$, so inequality (\ref{eqn:chi123bound}) is satisfied for these characters.  For the remaining characters with $\psi$ nontrivial, note that $\int |g|\, dm=1$, so (\ref{eqn:fStargHat}) implies $|\widehat{f*_{\Gamma}g}|\leq |\hat{f}|$ everywhere. Now (\ref{eqn:OtherChars}) and (\ref{eqn:GammaPsi0}) imply (\ref{eqn:chi123bound}). To prove (\ref{eqn:SamePrime}), write
\begin{align*}
  (f*_\Gamma g)'(x) &= \int \int f(x,y+w_1) g(w_2)\, dm_{\Gamma}(w)\, dy\\
   &= \int \int f(x,y+w_1) \, dy \,  g(w_2)\, dm_{\Gamma}(w)\\
   &=  f'(x) \int g(w_2)\, dm_{\Gamma}(w)\\
   &= f'(x). \qedhere
\end{align*}

\end{proof}

\begin{lemma}\label{lem:IntegrateToKronecker}
	With the hypotheses of Lemma \ref{lem:BHconvolvesToUniform}, let $f: \mathbb T^d\times \mathbb T^d\to [0,1]$, define $f': \mathbb T^d\to [0,1]$ by  $f'(x):=\int_{\mathbb T^d} f(x,y)\, dy$, and let $J':=  \int \int f'(x)f'(x+s)f'(x+2s)  \, ds \, dx$. 	Then there is a cylinder function $g$ subordinate to $U$ such that
\[
	J:=\int f(x,y) f(x+s,y+t) f*_{\Gamma}g(x+2s, y+2t) \, ds\, dt \, dx\, dy.
\]
	satisfies $|J-J'|<k^{-1/2}$.
\end{lemma}

\begin{proof}   Choose, by Lemma \ref{lem:BHconvolvesToUniform}, a cylinder function $g$ subordinate to $U$ so that
	\begin{align}
		\label{eqn:2GgBound} &|\widehat{f*_\Gamma g}(\chi,\psi)| < k^{-1/2} \qquad \text{for all } \chi, \psi  \text{ with } \psi \in \widehat{\mathbb T}^d\setminus \{0\},\\
		&(f*_{\Gamma}g)'=f'.
 		\end{align}
Now we apply Lemma \ref{lem:SmallFourierToW} with $Z =  \mathbb T^d\times \mathbb T^d$,  $W = \mathbb T^d$, $K= \{0\}\times \mathbb T^d$, $I= J$, and $I_W = \int f'(x)f'(x+s)(f*_{\Gamma} g)'(x+2s) \, dx\, ds$,  using $\kappa = k^{-1/2}$ as supplied by inequality (\ref{eqn:2GgBound}).  We conclude that $|J-I_W|<k^{-1/2}$. Since $(f*_{\Gamma}g)'=f'$, we have $I_W=J'$, so this is the desired conclusion.  \end{proof}

\section{Averages for standard 2-step Weyl systems}\label{sec:AffineLimits}

In the next section we will prove Lemma \ref{lem:BHKroneckerAlternative} by reducing the general statement to the special case where the totally ergodic system under consideration is a standard 2-step Weyl system.  Proposition \ref{prop:WeakLimit} will then allow us to compute the limit of the multiple ergodic averages appearing in (\ref{eqn:RothApproxWithg}).

	For the remainder of this section, we fix $d, r\in \mathbb N$, $\alpha\in \mathbb T^d$, $\beta \in \mathbb T^r$ and let $S:(\mathbb T^d)^2\to (\mathbb T^d)^2$ be given by $S( x, y) = ( x +  2\alpha,  y +  x)$.  We assume $\alpha$ and $\beta$ generate $\mathbb T^d$ and $\mathbb T^r$, respectively, and we write $m$ for Haar probability measure on $\mathbb T^d$.  We maintain the notational conventions introduced in \S\ref{sec:Outline} and the intervening sections.

\begin{proposition}\label{prop:WeakLimit}
  With $d, r, \alpha, \beta$, and $S$ defined above, there is an affine joining $\Gamma$ of $\mathbb T^d$ with $\mathbb T^r$ such that for all Riemann integrable $g:\mathbb T^r\to \mathbb R$ and all bounded measurable $f: \mathbb T^d\times \mathbb T^d\to \mathbb R$ we have
  \begin{equation}\label{eqn:WeakLimit}
    \begin{split}
      \lim_{N\to\infty} \frac{1}{N}\sum_{n=1}^N g(n^2\beta)&\int f \cdot f\circ S^n\cdot f\circ S^{2n} \, d(m\times m)  \\
      &=\int f(x,y) f(x+s,y+t) f*_{\Gamma}g(x+2s, y+2t) \, ds\, dt \, dx\, dy,
    \end{split}
  \end{equation}
  where $f*_\Gamma g$ defined in (\ref{def:FstarG}).
\end{proposition}

\begin{remark}  We use ``$2\alpha$'' in place of ``$\alpha$'' in our definition of $S$ to simplify computations.  Since every generating $\alpha\in \mathbb T^d$ can be written as $2\alpha'$, where $\alpha'$ is generating, there is no loss of generality.
\end{remark}

We first prove Lemma \ref{lem:IntegralFormula}, which provides explicit limits of polynomial averages on $(\mathbb T^d)^4\times \mathbb T^r$.  Lemma \ref{lem:ExplicitLimit} then provides an explicit pointwise-a.e.~limit for the relevant averages in (\ref{eqn:WeakLimit}) when  $f$ and $g$ are continuous.   Corollary \ref{cor:ExplicitL2} uses these to establish $L^2$ convergence with the same limit formula, for bounded measurable $f$ and Riemann integrable $g$.  Proposition \ref{prop:WeakLimit} is then proved in the last paragraph of this section.

The following lemma is needed for the proof of Lemma \ref{lem:IntegralFormula}; it is nothing but Fubini's theorem together with the translation invariance of Haar measure.
\begin{lemma}\label{lem:Fubini}
Let $\nu$ be a Borel probability measure on $\mathbb T^d\times \mathbb T^r$, let $m$ be Haar measure on $\mathbb T^d$, and let $h:\mathbb T^d\times (\mathbb T^d\times \mathbb T^r)\to \mathbb C$ be continuous.  Then
\[
\int h(t,w)\, d\nu(w)\, dm(t) = \int h(t-\pi_1(w), w) \, d\nu(w)\, dm(t)
\]
where $\pi_1 :\mathbb T^d\times \mathbb T^r\to \mathbb T^d$ is the projection map.
\end{lemma}

Let $G=(\mathbb T^d)^4\times \mathbb T^r$, with elements of $G$ written $(z_1,z_2,z_3,z_4,z_5)$, $z_i\in \mathbb T^d$ for $i\leq 4$, $z_5\in \mathbb T^r$.  Let $G_{3AP}$ be the closed connected subgroup $\{(s,t,2s,2t,0):s,t\in \mathbb T^d\}\subseteq G$.

\begin{lemma}\label{lem:IntegralFormula}
With $\alpha$, $\beta$, $d$, $r$, and $G$ as above, let $\mb u = (0,\alpha,0,4\alpha,\beta)\in G$. Then there is an affine joining $\Gamma$ of $\mathbb T^d$ with $\mathbb T^r$ such that
\begin{equation}\label{eqn:ANIntegral}
\lim_{N\to\infty} \frac{1}{N}\sum_{n=1}^N F(n\mb c + n^2\mb u) = \int F(s,t,2s,2t+2w_1,w_2)\, dm_{\Gamma}(w) \, ds\, dt
\end{equation}
for every continuous $F:G\to \mathbb C$ and all $\mb c\in G$ such that $\overline{\langle \mb c \rangle} = G_{3AP}$.
\end{lemma}

\begin{proof}  Assume $\mb c\in G$ is such that $\overline{\langle \mb c \rangle} = G_{3AP}$. Let $\Lambda = \overline{\langle \mb u\rangle}$ and let $\Phi = \overline{\langle \mb c, \mb u\rangle}$.  Note that $\Phi = G_{3AP}+\Lambda$.  Also, $\Phi$ does not depend on $\mb c$ (asssuming $\mb c$ generates $G_{3AP}$).

Since $\Phi$ is a closed subgroup of a finite dimensional torus, its identity component $\Phi_0$ has finite index in $\Phi$.  Part (ii) of Lemma \ref{lem:ConnectedWeyl} then provides cosets $\Phi_j$ of $\Phi_0$ in $\Phi$ and nonnegative $c_j$ with $\sum c_j=1$ such that for every continuous $F:G\to \mathbb C$ we have
\begin{equation}\label{eqn:JoinPhi}
  \lim_{N\to\infty} \frac{1}{N}\sum_{n=1}^N F(n\mb c + n^2\mb u) = \sum c_j \int F\, dm_{\Phi_j}.
\end{equation}
 Part (iii) of Lemma \ref{lem:ConnectedWeyl} implies that the $c_j$ do not depend on $\mb c$, assuming $\overline{\langle \mb c\rangle}=G_{3AP}$ (which is connected). We will prove that for each coset $\Phi_j$ of $\Phi_0$ in $\Phi$, we can write
\begin{equation}\label{eqn:SimplifyPhij}
\int F\, dm_{\Phi_j} = \int F(s,t,2s,2t+2w_1,w_2) \, dm_{\Lambda_j}(w) \, ds\, dt
\end{equation}
where $\Lambda_j$ is a coset of $\Lambda_0$ ($=$ the identity component of $\Lambda$), and that $\Lambda_0$ can be viewed as a joining of $\mathbb T^d$ with $\mathbb T^r$.  Combining (\ref{eqn:SimplifyPhij}) with (\ref{eqn:JoinPhi}), we get (\ref{eqn:ANIntegral}), where $\Gamma$ is the affine joining of $\mathbb T^d$ with $\mathbb T^r$ determined by $c_j$ and $\Lambda_j$.

\begin{claim}
\begin{enumerate}
\item[(i)] $\Lambda$ is a joining of the closed subgroups $H_1:=\{(0,z,0,4z,0):z\in \mathbb T^d\}$ and $H_2:=\{(0,0,0,0,v):v\in \mathbb T^r\}$.  Its identity component $\Lambda_0$ is also a joining of $H_1$ and $H_2$.

\item[(ii)]  $\Phi_0:= G_{3AP}+\Lambda_0$ is the identity component of $G_{3AP}+\Lambda$.

\item[(iii)] Every coset of $\Phi_0$ in $\Phi$ has the form $G_{3AP}+\Lambda_j$ where $\Lambda_j$ is a coset of $\Lambda_0$ in $\Lambda$.
\end{enumerate}
\end{claim}
Part (i) of the Claim follows from Observation \ref{obs:Joining}, the fact that  $\alpha$ and $\beta$ are generating, and Observation \ref{obs:ComponentOfJoining}.

To prove part (ii), note that $G_{3AP}$ is closed and connected, so $G_{3AP}+\Lambda_0$ is a closed connected subgroup of $G_{3AP}+\Lambda$.  Since $\Lambda_0$ is the identity component of $\Lambda$, which is a closed subgroup of a finite dimensional torus, we see that $\Lambda_0$ has finite index in $\Lambda$.  Thus $G_{3AP}+\Lambda_0$ is a closed, connected, finite index subgroup of $G_{3AP}+\Lambda$, and therefore is its identity component.  Part (iii) is an immediate consequence of (ii) and the fact that $G_{3AP}$ is connected.

The Claim allows us to write integrals with respect to Haar measure over $\Phi_j$ explicitly.  We write integration over a coset $\Lambda_j$ of $\Lambda_0$ in $\Lambda$ as
\begin{equation}\label{eqn:Gamma1Integral}
	\int F\, dm_{\Lambda_j} = \int F(0,w_1,0,4w_1, w_2)\, dm_{\Lambda_j}(w).
\end{equation}
where the $m_{\Lambda_j}$ on the right is viewed as Haar probability measure on a coset of a joining of $\mathbb T^d$ with $\mathbb T^r$; this identification is possible as $H_1$ and $H_2$ are isomorphic to $\mathbb T^d$ and $\mathbb T^r$, respectively. We then write integration over $\Phi_j$ (=$G_{3AP}+\Lambda_j$) as
\begin{equation}\label{eqn:overPhi0}
\int F \, dm_{\Phi_j}=\int F(s,t+w_1,2s,2t+4w_1,w_2) \, dm_{\Lambda_j}(w) \, ds\, dt.
\end{equation}
This is justified by the fact that the above integral is invariant under translation by elements of $G_{3AP}$ and by elements of $\Lambda_0$, so the above integral is indeed integration with respect to Haar probability measure on $G_{3AP}+\Lambda_j$.

We may replace $t$ with $t-w_1$ in (\ref{eqn:overPhi0}).  To see this, first observe that the order of the outer integrals can be changed to $dt\, ds$.  For a fixed $s\in \mathbb T^d$ define $h_s$ on $\mathbb T^d \times (\mathbb T^d\times \mathbb T^r)$ by $h_s(t,w):=F(s,t+w_1,2s,2t+4w_1,w_2)$.  The right hand side of (\ref{eqn:overPhi0}) can then be written as $\int \int h_s(t,w)\, dm_{\Lambda_j}(w)\, dt\, ds$. We apply Lemma \ref{lem:Fubini} with $m_{\Lambda_0}$ in place of $\nu$, and again change the order of integration. The integral in (\ref{eqn:overPhi0}) therefore simplifies to yield (\ref{eqn:SimplifyPhij}), completing the proof.
\end{proof}

An immediate consequence of Lemma \ref{lem:IntegralFormula} is that for every continuous $F:G\to \mathbb C$
\begin{equation}\label{eqn:z}
\lim_{N\to\infty} \frac{1}{N}\sum_{n=1}^N F(\mb z + n\mb c +n^2\mb u) = \int F(z_1+s,z_2+t,z_3+ 2s, z_4 + 2t+2w_1, z_5+w_2)\, dm_{\Gamma}(w) \, ds\, dt
\end{equation}
for all $\mb z\in G$, all $\mb c\in G$ such that $\overline{\langle \mb c\rangle}=G_{3AP}$.  This can be seen by applying Lemma \ref{lem:IntegralFormula} with the translate $F_{-\mb z}$ in place of $F$.

\begin{lemma}\label{lem:ExplicitLimit}
  With the above $d, r, \alpha, \beta$,  $S$, $G$, and the affine joining $\Gamma$ provided by Lemma \ref{lem:IntegralFormula}, there is a set $W\subseteq \mathbb T^d$ with $m(W)=1$ such that
\begin{equation}\label{eqn:AffineAverageDef}
	\begin{split}\lim_{N\to \infty}  \frac{1}{N}\sum_{n=1}^N &g(n^2\beta) \cdot f\circ S^{n}(x,y) \cdot f\circ S^{2n}(x,y)  \\
	&= \int f(x+s, y+ t)f(x+2s, y+2t+2w_1)g(w_2) \, dm_{\Gamma}(w) \, ds\, dt.\end{split}
\end{equation}
for all $x\in W$, all $y\in \mathbb T^d$, and all continuous $f:(\mathbb T^d)^2\to \mathbb R$, $g: \mathbb T^r\to \mathbb R$.
\end{lemma}

\begin{proof}  Write the terms the left hand side of (\ref{eqn:AffineAverageDef}) as
\begin{align*}
g(n^2\beta)f(S^{n}(x,y))&f(S^{2n}(x,y))\\&= F(x+2n\alpha,y+nx+\tbinom{n}{2}2\alpha,x+4n\alpha,y+2nx+\tbinom{2n}{2}2\alpha,n^2\beta)\\
&= F(\mb z_{x,y} + n\mb c_x + n^2\mb u),
\end{align*}
where $F:G\to \mathbb R$ is given by $F(x,y,x',y',z):=f(x,y)f(x',y')g(z)$ and \[\mb z_{x,y} = (x,y,x,y,0) \qquad \mb c_{x} = (2\alpha, x+\alpha, 4\alpha, 2(x+\alpha),0) \qquad \mb u = (0,\alpha,0,4\alpha,\beta).\]
If $f$ and $g$ are continuous then $F$ is continuous, and we may apply Lemma \ref{lem:IntegralFormula} (and (\ref{eqn:z}) in particular) to conclude that
\[
\lim_{N\to\infty} \frac{1}{N}\sum_{n=1}^N F(\mb z_{x,y} + n\mb c_x + n^2\mb u)= \int F(x+s,y+t,x+2s,y+2t+2w_1,w_2) \, dm_{\Gamma}(w) \, ds\,dt
\]
for all $x\in \mathbb T^d$ such that $\overline{\langle \mb c_x \rangle}=G_{3AP}$ and all $y\in \mathbb T^d$.   This is equivalent to (\ref{eqn:AffineAverageDef}).

Let $W:=\{x\in \mathbb T^d: \overline{\langle \mb c_x \rangle}=G_{3AP}\}$. To complete the proof, we will show that $m(W)=1$.  Note that $x\in W$ if and only if $\chi(\mb c_x)\neq 1$ for every nontrivial character $\chi$ of $G_{3AP}$, and there are only countably many characters, so it suffices to prove that for every such $\chi$, $m(E_\chi)=0$, where $E_\chi:=\{x\in \mathbb T^d: \chi(\mb c_x)=1\}$.  Every character of $G_{3AP}$ can be written as $\chi((s,t,2s,2t,0))= \exp(2\pi i (\mb j\cdot s + \mb k \cdot t))$ for some $\mb j, \mb k\in \mathbb Z^d$, and if $\chi$ is nontrivial then $\mb j$ and $\mb k$ are not both $0$.  Thus, $\chi(\mb c_x)=1$ if and only if $\mb k \cdot x = -(2\mb j + \mb k)\cdot \alpha$.  When $\mb k=\mb 0$ and $\mb j\neq \mb 0$, we then have $E_{\chi}=\varnothing$.  When $\mb k \neq \mb 0$, we see that $E_{\chi}$ is contained in a coset of the closed proper subgroup $\{x\in \mathbb T^d: \mb k\cdot x = 0\}$, so that $m(E_\chi)=0$. \end{proof}

\begin{corollary}\label{cor:ExplicitL2}
  With $d,r,\alpha,\beta$ and $S$ defined above, let $f\in L^\infty(m\times m)$ and let $g:\mathbb T^r\to \mathbb R$  be Riemann integrable. Define $A_N\in L^\infty(m\times m)$ by
  \[A_N:=  \frac{1}{N}\sum_{n=1}^N g(n^2\beta) \cdot f\circ S^{n} \cdot f\circ S^{2n}\]
   and let $A(x,y):=\int f(x+s, y+ t)f(x+2s, y+2t+2w_1)g(w_2) \, dm_{\Gamma}(w) \, ds\, dt$, where $\Gamma$ is the affine joining given by Lemma \ref{lem:IntegralFormula}.  Then $\lim_{N\to\infty} A_N = A$ in $L^2(m\times m)$.
\end{corollary}

\begin{proof} Let $W$ be the set provided by Lemma \ref{lem:ExplicitLimit}.  To deduce Corollary \ref{cor:ExplicitL2}, we first prove that (\ref{eqn:AffineAverageDef}) holds for all $x\in W$, $y\in \mathbb T^d$, assuming $f$ is continuous and $g$ is Riemann integrable.  To prove this, assume we have such $x$, $y$, $f$, and $g$. Let $h_0^{(k)}, h_1^{(k)}$ be continuous functions on $\mathbb T^r$ satisfying  $\inf g \leq h_0^{(k)}\leq g \leq  h_1^{(k)}\leq \sup g$ pointwise, such that $\lim_{k\to\infty} \int h_1^{(k)} - h_0^{(k)} \, dm_{\mathbb T^r}=0$.  For each $k$, Lemma \ref{lem:IntegralFormula} says that (\ref{eqn:AffineAverageDef}) holds with $h_i^{(k)}$ in place of $g$.   Applying Lemma \ref{lem:RiemannCoefficients} with $f\circ S^n(x,y)\cdot f\circ S^{2n}(x,y)$ in place of $v_n$, we see that
\begin{equation}\label{eqn:h0}\lim_{N\to\infty} A_N(x,y) = \lim_{k\to\infty} \int f(x+s, y+ t)f(x+2s, y+2t+2w_1)h_0^{(k)}(w_2) \, dm_{\Gamma}(w) \, ds\, dt
\end{equation}
The pointwise inequalities $h_0^{(k)} \leq g \leq h_1^{(k)}$ and the assumption $\lim_{k\to\infty} \int h_1^{(k)} - h_0^{(k)}\, dm_{\mathbb T^d}=0$ now imply that $\lim_{k\to\infty} h_0^{(k)} = g$ in $L^2(m_{\mathbb T^r})$.  The limit on the right of (\ref{eqn:h0}) is therefore equal to $\int f(x+s,y+t)f(x+2s,y+2t+2w_1) g(w_2)\, dm_\Gamma(w)\, ds\, dt$.  This proves that $A_N$ converges to $A$ for $m\times m$ almost every $(x,y)$, and the dominated convergence theorem then implies $A_N$ converges to $A$ in $L^2(m\times m)$, in the special case where $f$ is continuous and $g$ is Riemann integrable.

To prove the general case of Corollary \ref{cor:ExplicitL2}, let $f\in L^\infty(m\times m)$ and let $\varepsilon>0$.  Assume, without loss of generality, that $\sup(|g|)\leq 1$. We write $\|\cdot\|$ for the $L^2$ norm given by $m\times m$. Let $f_0:(\mathbb T^d)^2\to \mathbb R$ be continuous with $\|f-f_0\|<\varepsilon$.

Let $A_N':=\frac{1}{N}\sum_{n=1}^N g(n^2\beta)\cdot f_0\circ S^n \cdot f_0\circ S^{2n}$, and let
\[A'(x,y):=\int f_0(x+s, y+ t)f_0(x+2s, y+2t+2w_1)g(w_2) \, dm_{\Gamma}(w) \, ds\, dt.\]  Note that\footnote{Apply the identity $ab-cd =  a(b-d)+(a-c)(d)$ with $a=f_0\circ S^n$, $b=f_0\circ S^{2n}$, $c=f\circ S^n$, $d=f\circ S^{2n}$, note that $\|f\circ S^n - f_0\circ S^n\|=\|f-f_0\|$ and likewise for $S^{2n}$.}
\[\|f\circ S^n\cdot f\circ S^{2n} - f_0 \circ S^n \cdot f_0\circ S^{2n}\|\leq 2\|f-f_0\|,\] hence $\|A_N-A_N'\|\leq 2(\sup |g|)\|f-f_0\|$ for every $N$, and  similarly $\|A-A'\|\leq 2(\sup |g|)\|f-f_0\|$.  Then \begin{align*}
\|A_N-A\| &=  \|A_N-A_N'+A_N' - A' + A'- A\|\\
          &\leq \|A_N-A_N'\| + \|A_N' - A'\| + \|A'- A\|\\
          &< 2\varepsilon + \|A_N' - A'\| + 2\varepsilon && \text{assuming } \sup(|g|)\leq 1.
\end{align*}
From the first paragraph of this proof, we have $\|A_N'-A'\|\to 0$ as $N\to\infty$.  Combining this with the above inequalities, we get $\limsup_{N\to\infty}\|A_N-A\| \leq 4\varepsilon$.  Since $\varepsilon$ was arbitrary and $g$ is fixed, we have $\|A_N-A\|\to 0$ as $N\to\infty$. \end{proof}

Proposition \ref{prop:WeakLimit} now follows from Corollary \ref{cor:ExplicitL2}, observing that the left hand side of (\ref{eqn:WeakLimit}) is $\lim_{N\to\infty} \int f(x,y) A_N(x,y)\, dx\, dy$, with $A_N$ as in Corollary \ref{cor:ExplicitL2}, and the right hand side of $(\ref{eqn:WeakLimit})$ can be written as $\int f(x,y) A(x,y) \, dx\, dy$ (recalling the definition of $f*_{\Gamma} g$ from (\ref{def:FstarG})).

\section{Proof of Lemma \ref{lem:BHKroneckerAlternative}}\label{sec:MainProof}
Recall the statement of Lemma \ref{lem:BHKroneckerAlternative}: let $k<r\in \mathbb N$, $\ell\in \mathbb N$, let $\bm\beta\in \mathbb T^r$ be generating, and let $U\subseteq \mathbb T^r$ be an approximate Hamming ball of radius $(k,\eta)$ for some $\eta>0$. For every totally ergodic MPS $(X,\mathcal B,\mu,T)$ and every measurable $f:X\to [0,1]$, there is a cylinder function $g=\frac{1}{m(V)}1_V$ subordinate to $U$ such that
	\begin{equation}\label{eqn:RothApproxWithgAgain}
		\lim_{N\to \infty} \Bigl|\frac{1}{N}\sum_{n=1}^N g(n^2\ell^2\bm\beta)\int f\cdot T^{n} f \cdot T^{2n} f\, d\mu - L_3(f,T)\Bigr|<2k^{-1/2}\|f\|^2.
	\end{equation}

\begin{proof}Let $(X,\mathcal B,\mu,T)$ be a totally ergodic MPS,  $f:X\to [0,1]$, and $k\in \mathbb N$.
Let $BH$ be a proper Bohr-Hamming ball of radius $(k,\eta)$ for some $\eta>0$.  Write $BH$ as $\{n:n\bm\beta\in U\}$, for some  approximate Hamming ball $U\subseteq \mathbb T^r$ of radius $(k,\eta)$ and some generating $\bm\beta\in \mathbb T^r$.  Note that $\ell^2\bm\beta$ is also generating. For a Riemann integrable $g:\mathbb T^r\to \mathbb R$, write
\[
	A(f,g):= \lim_{N\to\infty}\frac{1}{N}\sum_{n=1}^N g(n^2\ell^2\bm\beta)\int f\cdot T^n f\cdot T^{2n} f\, d\mu.
\]
We will prove that there is a cylinder function $g$ subordinate to $U$ such that
\begin{equation}\label{eqn:Goal}
  |A(f,g) - L_3(f,T)|<2k^{-1/2}\|f\|^2.
\end{equation}
Let $M=\frac{1}{m(V)}$, where $V$ is one of the cylinders $V_{I,\mb y,\eta}$ in (\ref{eqn:UunionV}).  In other words, $M = \|g\|_{\infty}$ for each cylinder function $g$ subordinate to $U$.  Choose, by Lemma \ref{lem:ReductionToWeyl}, a factor $\pi:\mb X\to \mb Y=(Y,\mathcal D,\nu,S)$ so that $\mb Y$ is a factor of a standard 2-step Weyl system, and such that for all Riemann integrable $g:\mathbb T^r\to [0,M]$, we have
\begin{equation}\label{eqn:Ycharacteristic}
	|A(f,g)- B(\tilde{f},g)|<\tfrac{1}{2}k^{-1/2}\|f\|^2.
\end{equation}
where $\tilde{f}\circ \pi =P_{\mb Y}f$ and $B(\tilde{f},g):= \lim_{N\to\infty}\frac{1}{N}\sum_{n=1}^Ng(n^2\ell^2\bm\beta)\int \tilde{f} \cdot S^{n}\tilde{f} \cdot S^{2 n}\tilde{f}\, d\nu$. Let
\[C(\tilde{f}):= \lim_{N\to\infty}\frac{1}{N}\sum_{n=1}^N \int \tilde{f}\cdot S^n \tilde{f}\cdot S^{2n}\tilde{f}\, d\nu,\] so that $C(\tilde{f})=B(\tilde{f},\mb 1)$. Note that $A(f,\mb 1)= L_3(f,T)$, so the special case of (\ref{eqn:Ycharacteristic}) with $g= \mb 1$ yields
\begin{equation}\label{eqn:CminusD}
|C(\tilde{f})-L_3(f,T)|<\tfrac{1}{2}k^{-1/2}\|f\|^2.
\end{equation}
Let $\tilde{\mb Y} = (\tilde{Y},\tilde{\mathcal D},\tilde{\nu},\tilde{S})$ be an extension of $\mb Y $ which is a standard 2-step Weyl system $(\mathbb T^d\times \mathbb T^d, \mathcal B_{\mathbb T^d\times \mathbb T^d}, m, \tilde{S})$, and view $\tilde{f}$ as a function on $\tilde{Y}=\mathbb T^d\times \mathbb T^d$ (cf.~Remark \ref{rem:Extensions}). By Proposition \ref{prop:WeakLimit}, there is an affine joining $\Gamma$ of $\mathbb T^d$ with $\mathbb T^r$ such that for each Riemann integrable $g:\mathbb T^r \to \mathbb R$ we have
\[ B(\tilde{f},g) = \int \tilde{f}(x,y)\tilde{f}(x+s,t+y) \tilde{f} *_{\Gamma} g (x+2x,y+2t) \, ds\, dt \, dx\, dy.\]
Let $J$ denote the integral above, define $\tilde{f}':\mathbb T^d\to [0,1]$ by $\tilde{f}'(x):=\int \tilde{f}(x,y)\, dy$, and let
\[
J':= \int \tilde{f}'(x)\tilde{f}'(x+s)\tilde{f}'(x+2s)\, dx\, ds.
\]
Choose, by Lemma \ref{lem:IntegrateToKronecker}, a cylinder function $g$ subordinate to $U$ so that
\begin{equation}\label{eqn:JJprime}
  |J-J'|<k^{-1/2}\|\tilde{f}\|^2.
\end{equation}  Observation \ref{obs:KroneckerWeyl} means $J'= C(\tilde{f})$, so  (\ref{eqn:JJprime}) can be written as
\begin{equation}\label{eqn:BminusC}
|B(\tilde{f},g)-C(\tilde{f})|<k^{-1/2}\|\tilde{f}\|^2.
\end{equation}
Combining (\ref{eqn:BminusC}) with (\ref{eqn:CminusD}), (\ref{eqn:Ycharacteristic}), and the triangle inequality, we get (\ref{eqn:Goal}), completing the proof.
\end{proof}

\section{Auxiliary lemmas }\label{sec:Appendix}

In \S\ref{sec:Compactness} we prove Lemma \ref{lem:RecurrenceCompactness}, essentially by repeating a routine proof of Furstenberg's correspondence principle.  Section \ref{sec:2StepAffine} explains a fact needed in the proof of Lemma \ref{lem:ReductionToAffine}, and \S\ref{sec:Markov} states two immediate consequences of Markov's inequality needed in the proof of Lemma \ref{lem:SqrtBHisrecurrent}.

\subsection{Compactness}\label{sec:Compactness} Here we write $[N]$ for the interval $\{0,1,\dots,N-1\}$ in $\mathbb Z$.
\begin{lemma}\label{lem:RecurrenceEquivalence}
 Let $S\subseteq \mathbb Z$ and $\delta\geq 0$.  The following conditions are equivalent.

  \begin{enumerate}
    \item[(i)]  There is a measure preserving system $(X,\mathcal B,\mu,T)$ and $A\subseteq X$ with $\mu(A)>\delta$ such that $\mu\bigl(\bigcap_{j=0}^k T^{-js}A\bigr)=0$ for all $s\in S$.

    \item[(ii)] $S$ is $(\delta,k)$-nonrecurrent, meaning (i) holds with $\bigcap_{j=0}^k T^{-js}A=\varnothing$ in place of $\mu\bigl(\bigcap_{j=0}^k T^{-js}A\bigr)=0$.

    \item[(iii)] There is a $\delta'>\delta$ such that for all $N\in \mathbb N$, there is a set $B_N\subseteq [N]$ with $|B_N|\geq \delta' N$ such that $\bigcap_{j=0}^{k} (B_N-js)=\varnothing$ for all $s\in S$.
  \end{enumerate}
\end{lemma}

\begin{proof}
To prove (i) implies (ii), let $A$ satisfy condition (i), and let $A':=A\setminus \bigcup_{s\in S} \bigcap_{j=0}^k T^{-js}A$.  Then $\mu(A')=\mu(A)>\delta$, while $A'\subseteq \bigcap_{j=0}^k T^{-js}A' \subseteq \bigcap_{j=0}^k T^{-js}A$ for every $s\in S$. Since $A'$ is both a subset of and disjoint from $\bigcap_{j=0}^k T^{-js}A$, we have $\bigcap_{j=0}^k T^{-js}A'=\varnothing$ for every $s\in S$.

To prove (ii) implies (iii),  suppose $A$ satisfies condition (ii). Let $\delta'$ be such that $\mu(A)>\delta'>\delta$. Fixing $x\in X$ and setting $A_x:=A\cap \{T^nx:n\in \mathbb N\}$, we have $\bigcap_{j=0}^{k} T^{-jn}A_x=\varnothing$. Setting $B_x:=\{n\in \mathbb Z:T^nx\in A\}$, we have $\bigcap_{j=0}^{k} (B_x-jn) = \{n\in \mathbb Z: T^nx \in \bigcap_{j=0}^k T^{-jn}A\}$.  Thus $\bigcap_{j=0}^{k} (B_x-jn)=\varnothing$ whenever $\bigcap_{j=0}^k T^{-jn}A=\varnothing$.

Set $F_N:=\frac{1}{N}\sum_{n=0}^{N-1} 1_A(T^nx)$.   Then $\int F_N(x)\, d\mu(x) = \mu(A).$
It follows that there is an $x\in X$ such that $F_N(x) \geq \delta'$.  Our definition of $F_N$ then implies $|B_x\cap [N]| \geq \mu(A)N$.

To prove (iii) implies (i), suppose condition (iii) holds.  Let $X=\{0,1\}^\mathbb Z$ with the product topology, and let $\mathcal B$ be the corresponding Borel $\sigma$-algebra.  Let $T:X\to X$ be the left shift, meaning $(Tx)(n)=x(n+1)$.  We will construct a Borel probability measure $\mu$ on $(X,\mathcal B)$ and find a clopen set $A\subseteq X$ satisfying (i).

Let $A:=\{x\in X:x(0)=1\}$ (so $A$ is the cylinder set where $1$ appears at index $0$).  For each $N\in \mathbb N$, let $y_N:=1_{B_N}\in X$.  Note that $1_A(T^ny_N)=1$ if and only if $n\in B_N$, and similarly
\begin{align}\label{eqn:BNintersect}
  1_{A \cap T^{-s}A \cap \cdots \cap T^{-ks}A} (T^n y_N) = 1 \quad \text{if and only if} \quad n\in \bigcap_{j=0}^{k} (B_N-js).
\end{align}
Form a measure $\mu_N$ on $X$ defined by
\[
\int f\, d\mu_N:= \frac{1}{N}\sum_{n=0}^{N-1}f(T^ny_N).
\]
Let $\mu$ be a weak$^*$ limit of the $\mu_N$ (i.e. choose a convergent subsequence of $\mu_N$ and let $\mu$ be the limit).  To see that $\mu$ is $T$-invariant, note that
\[
\Bigl|\int f\circ T\, d\mu_N - \int f\, d\mu_N\Bigr| =\frac{1}{N}|f(T^Ny_N) - f(y_N)|\leq \frac{2}{N}\sup |f|
\]
for every $N$, so $\int f\circ T\, d\mu = \int f\, d\mu$ for every bounded continuous $f$.  In particular, $\mu(T^{-1}C)=\int 1_C\circ T\, d\mu = \int 1_C\, d\mu = \mu(C)$ for every clopen set $C\subseteq X$.  Since the clopen subsets of $X$ generate the Borel $\sigma$-algebra of $X$, this proves that $T$ preserves $\mu$.

To see that $\mu(A)\geq \delta'$, note that
\[\mu(A) \geq \liminf_{N\to\infty} \frac{1}{N}\sum_{n=0}^{N-1}1_A(T^ny_N) \geq \liminf_{N\to\infty} \frac{1}{N}|B_N|\geq \delta'.\]
To prove that $\mu\Bigl(\bigcap_{j=0}^{k}T^{-js}A\Bigr)=0$ for all $s\in S$, fix $s\in S$ and note that (\ref{eqn:BNintersect}) implies
\[
\mu_N\Bigl(\bigcap_{j=0}^{k}T^{-js}A\Bigr) = \frac{1}{N}\sum_{n=0}^{N-1} 1_{A\cap T^{-s}A\cap \cdots \cap T^{-ks}A}(T^ny_N) \leq \frac{1}{N}\Bigl|\bigcap_{j=0}^k(B_N-js)\Bigr|=0
\]
for all $N\in \mathbb N$.  Since $C:=\bigcap_{j=0}^k T^{-js} A$ is clopen and $\mu$ is a weak$^*$ limit of the $\mu_N$, we have $\mu(C)=\lim_{N\to\infty} \mu_N(C)=0$. \end{proof}

Recall the statement of Lemma \ref{lem:RecurrenceCompactness}: if $k\in \mathbb N$, $0\leq \delta<\delta'$ and $S\subseteq \mathbb Z$ is such that every finite subset of $S$ is $(\delta',k)$-nonrecurrent, then $S$ is $(\delta,k)$-nonrecurrent.

\begin{proof}[Proof of Lemma \ref{lem:RecurrenceCompactness}]

Suppose $S\subseteq \mathbb Z$, $k\in \mathbb N$, $0\leq \delta<\delta'$, and that every finite subset of $S$ is $(\delta',k)$-nonrecurrent. Applying Lemma \ref{lem:RecurrenceEquivalence} to the finite set $S_N:=S\cap [-N,N]$, we may choose, for each $N$, a set $B_N\subseteq [N]$ such that $|B_N|>\delta'N$ and $\bigcap_{j=0}^N (B_N-js)=\varnothing$ for all $s\in S\cap [-N,N]$.  Note that this implies $\bigcap_{j=0}^N (B_N-js)=\varnothing$ for all $s\in S$, since $B_N-js$ is disjoint from $[N]$ for every $s\in S\setminus [-N,N]$.  This means $S$ satisfies condition (iii) of Lemma \ref{lem:RecurrenceEquivalence}, and we conclude that $S$ is $\delta$-nonrecurrent.
\end{proof}

\subsection{The 2 step affine factor of a totally ergodic nilsystem}\label{sec:2StepAffine}
A \emph{nilsystem} is an MPS $(Y,\mathcal D,\nu,S)$ where $Y=G/\Gamma$, with $G$ a nilpotent Lie group and $\Gamma$ a cocompact discrete subgroup, $\mathcal D$ is the Borel $\sigma$-algebra of $Y$, $\nu$ is the unique probability measure on $(Y,\mathcal D)$ invariant under left multiplication, and $Sy = a y$ for some fixed $a\in G$.

When $G$ is a topological group, we write $G_0$ for the connected component of the identity.  For Lie groups, $G_0$ is a closed subgroup of $G$.   We will use the fact that an ergodic nilsystem $(G/\Gamma,\mathcal B,\mu,T)$ is totally ergodic if and only if $G/\Gamma$ is connected.

 Lemma \ref{lem:2AofY} identifies the maximal $2$-step affine factor of a totally ergodic nilsystem; the purpose of this subsection is to explain how it follows from the results of \cite{FrantzikinakisThreePoly}, where it is essentially proved but not explicitly stated.

\begin{lemma}\label{lem:2AofY}
  Let $\mb X=(X,\mathcal B,\mu,T)$ be a totally ergodic nilsystem.  The maximal $2$-step affine factor $\mb A_2(\mb X)$ of $\mb X$ is isomorphic to $(\mathbb T^d,\mathcal B,m,A)$, where $d\in \mathbb N$ and $A:\mathbb T^d\to \mathbb T^d$ is a $2$-step unipotent affine transformation.
\end{lemma}

We will use the following standard fact about factors: let $\pi_i:\mb X\to \mb X_i = (X_i,\mathcal B_i,\nu_i,T_i),$  $i=1,2$, be two factors of a system where $(X_i,\mathcal B_i,\nu_i)$ are separable as measure spaces. Then $\mb X_1$ and $\mb X_2$ are isomorphic (as measure preserving systems) if the algebra of bounded $\mb X_1$-measurable functions is equal, up to $\mu$-measure $0$, to the algebra of bounded $\mb X_2$-measurable functions. We also need the following lemma from \cite{FrKrPolyAffine}.

\begin{lemma}[\cite{FrKrPolyAffine}, Proposition 3.1]
  Let $X=G/\Gamma$ be a connected nilmanifold such that $G_0$ is abelian.  Then any nilrotation $T_a(x)=ax$ defined on $X$ with Haar measure $\mu$ is isomorphic to a unipotent affine transformation $U$ on some finite dimensional torus.
\end{lemma}

\begin{remark}\label{rem:Step} The computation in \cite{FrKrPolyAffine} showing that the transformation $U$ is unipotent also shows that when $G$ is $k$-step nilpotent, $U$ is $k$-step unipotent.
\end{remark}

We now explain how Lemma \ref{lem:2AofY} follows from \cite{FrantzikinakisThreePoly}.  Let $\mb X$ be a totally ergodic nilsystem, $\mb X = (X,\mathcal B,\mu,T)$, where $X = G/\Gamma$, $G$ being a nilpotent Lie group, $\Gamma$ a cocompact lattice in $G$,  and $\mu$ the unique left-translation invariant Borel probability measure on $G/\Gamma$, $Tx\Gamma:=ax\Gamma$ for some fixed $a\in G$. Proposition 2.4 of \cite{FrantzikinakisThreePoly} shows that the algebra of functions measurable with respect to $\mb A_2(\mb X)$ coincides with the functions measurable with respect to the factor $\pi_2:\mb X\to \mb Y$, where $\mb Y = (X',\mathcal B', \mu',T')$,  $X':=G/(G_3[G_0,G_0]\Gamma)$, the factor map is given by $\pi_2(x\Gamma):= xG_3[G_0,G_0]\Gamma$, and $T'y = \pi_2(a)y$.   Furthermore, it is easy to verify (given the background suggested in \S2.2 of \cite{FrantzikinakisThreePoly}) that $X'$ can be written as $G'/\Gamma'$, where $\Gamma'$ is a cocompact lattice in $G':=G/(G_3[G_0,G_0])$, and $G'$ is a 2-step nilpotent Lie group with abelian identity component.   Proposition 3.1 of \cite{FrKrPolyAffine} (cf.~Remark \ref{rem:Step} above) says that $\mb Y$ is isomorphic to a $2$-step unipotent affine transformation $A$ on a finite dimensional torus.  Since the $\mb A_2(\mb X)$-measurable functions coincide with the $\mb Y$-measurable functions, we get that $\mb A_2(\mb X)$ is itself isomorphic to $\mb Y$.

\subsection{Consequences of Markov's inequality}\label{sec:Markov}

Let $(X,\mu)$ be a probability space partitioned into subsets $X_i$, $0\leq i \leq M-1$, with $\mu(X_i)=1/M$ for each $i$, and let $f:X\to [0,1]$ have $\int f\, d\mu>\delta$.  Let $f_i:=f1_{X_i}$.

\begin{lemma}\label{lem:Markov1}
With $X$, $f$, and $X_i$ specified above, let $I:=\{i: \int f_i\, d\mu> \frac{\delta}{2M}\}$.  Then $|I|> M\delta/2$.
\end{lemma}

\begin{proof}  Let $I':=\{0,\dots,M-1\}\setminus I$.  Note that
\[\delta<\int f\, d\mu = \sum_{i\in I'} \int f_i\, d\mu + \sum_{i\in I} \int f_i\, d\mu \leq  \sum_{i\in I'} \frac{\delta}{2M} + \sum_{i\in I} \frac{1}{M} = \frac{\delta}{2M}(M-|I|)+\frac{1}{M}|I|,
\] so $\delta < \frac{\delta}{2} + \frac{|I|}{M}\bigl(1-\frac{\delta}{2}\bigr)$.  This can be rearranged to $M\delta/2 < |I|(1-\delta/2)$, which implies $M\delta/2<|I|$.
\end{proof}

\begin{lemma}\label{lem:Markov2}
	With $X$ and $X_i$ as defined above, let $c, \varepsilon>0$ and assume $f, g:X\to \mathbb R$ satisfy $\|f-g\|_{L^1(\mu)}<\varepsilon$. Define
\[J := \Bigl\{i: \int_{X_i} |f-g|\, d\mu< \frac{c}{M} \Bigr\}.\]  Then $|J| > M(1-\frac{\varepsilon}{c})$.
\end{lemma}
\begin{proof}
	We estimate $J'$, where $J':=\{0,\dots,M-1\}\setminus J$.  Let $\varepsilon_i := \int_{X_i} |f-g|\, d\mu$.
	
	Note that $ \sum_{i=0}^{M-1} \varepsilon_i=\|f-g\|_{L^1(\mu)} < \varepsilon$, so $J'=\{i : \varepsilon_i \geq c/M\}$ satisfies $|J'|\cdot c/M<\varepsilon$, meaning  $|J'|<M\varepsilon/c$.  Thus $|J|=M-|J'|>M(1-\frac{\varepsilon}{c})$.
\end{proof}

The next lemma is an immediate consequence of the triangle inequality and the identity
\[
f_1 f_2\cdots f_k - h_1h_2\cdots h_k = \sum_{i=1}^k f_1\cdots f_{i-1}(f_i-h_i)h_{i+1}\cdots h_{k}.
\]
\begin{lemma}\label{lem:MultiBound}
  If $(X,\mu)$ is a probability space, $f_i, h_i: X\to [0,1]$, $i=1,\dots,k$, and $\|f_i-h_i\|_{L^1(\mu)}<\varepsilon$ for each $i$, then $|\int f_1 f_2\cdots f_k\, d\mu - \int h_1h_2\cdots h_k\, d\mu|<k\varepsilon$.
\end{lemma}

\subsection{Convergence with Riemann integrable coefficients}
Let $r\in \mathbb N$.  We say that a sequence $(y_n)_{n\in \mathbb N}$ of elements of $\mathbb T^r$ is \emph{uniformly distributed} if
\[\lim_{N\to\infty} \frac{1}{N}\sum_{n=1}^N g(y_n)=\int g\, dm\] for every continuous $g:\mathbb T^r\to \mathbb C$, where $m$ is Haar probability measure on $\mathbb T^d$.
\begin{lemma}\label{lem:RiemannCoefficients}
  Let $r\in \mathbb N$, let $(y_n)_{n\in \mathbb N}$ be a uniformly distributed sequence of elements of $\mathbb T^r$.  If $(v_n)_{n\in \mathbb N}$ is a bounded sequence of real numbers such that $L(g):=\lim_{N\to\infty} \frac{1}{N}\sum_{n=1}^N g(y_n)v_n$ exists for every continuous $g:\mathbb T^r\to \mathbb R$, then $L(g)$  exists for all Riemann integrable $g$.

   Furthermore, if $h_0^{(k)}, h_1^{(k)}$ are continuous functions on $\mathbb T^r$ with $h_0^{(k)}\leq g \leq h_1^{(k)}$ pointwise and $\lim_{k\to\infty} \int h_1^{(k)} -h_0^{(k)}\, dm=0$, then $L(g)=\lim_{k\to\infty} L(h_0^{(k)})=\lim_{k\to\infty} L(h_1^{(k)})$.

   Finally, if $C>0$ and $|L(g)|\leq C$ for every continuous $g:\mathbb T^r \to [0,1]$ then $|L(g)|\leq C$ for every Riemann integrable $g:\mathbb T^r\to [0,1]$.
\end{lemma}

\begin{proof}
Note that it suffices to prove the statement under the additional assumption that $v_n \in [0,1]$ for each $n$.  The general case follows by linearity.

   Let $g:\mathbb T^r\to \mathbb R$ be Riemann integrable.   Let $\varepsilon>0$, and choose continuous $g_0, g_1:\mathbb T^r\to \mathbb R$ so that $g_0\leq g \leq g_1$, $\int g_1-g_0\, dm<\varepsilon $.  Let $A_N(g):=\frac{1}{N}\sum_{n=1}^N g(y_n)v_n$.  We have
  \begin{align}\label{eqn:Lg0Lg1}
  L(g_0)=\lim_{N\to\infty} A_N(g_0) \leq \liminf_{N\to\infty} A_N(g) \leq \limsup_{N\to\infty} A_N(g) \leq \lim_{N\to\infty} A_N(g_1)=L(g_1)
  \end{align}
  and $L(g_1)-L(g_0)  = L(g_1-g_0)\leq \lim_{N\to \infty} \frac{1}{N}\sum_{n=1}^N g_1(y_n)-g_0(y_n) =\int g_1-g_0\, dm < \varepsilon$.  Since $\varepsilon>0$ was arbitrary, this proves that $A_N(g)$ converges, meaning $L(g)$ exists.

  A nearly identical argument will prove the second assertion of the lemma.  The third assertion follows from the second, by assuming $h_i^{(k)}:\mathbb T^r\to [0,1]$.
 \end{proof}

\section{Remarks}\label{sec:Remarks}

\subsection{More general \texorpdfstring{$2$}{2}-recurrence.} We say that $S\subseteq \mathbb Z$ is \emph{good for $k$-recurrence of powers} if for every MPS $(X,\mathcal B,\mu,T)$, every $A\subseteq X$ with $\mu(A)>0$ and all $c_1,\dots,c_k\in \mathbb N$, there is an $n\in S$ such that $A\cap T^{-c_1n}A\cap \dots \cap T^{-c_k n}A\neq \varnothing$.

 Problem 5 of \cite{FrProblems} asks whether $S\subseteq \mathbb Z$ being good for $k$-recurrence of powers implies $S^{\wedge k}$ is a set of measurable recurrence.  Our proof of Theorem \ref{thm:Main} does not immediately resolve this question for $k=2$, since we considered intersections of the form $A\cap T^{-n}A\cap T^{-2n}A$ (i.e. $c_1=1, c_2=2$ only).  We believe that our proof can be modified slightly to construct a set $S$ which is good for $2$-recurrence of powers such that $S^{\wedge 2}$ is not a set of measurable recurrence.

\subsection{Higher order recurrence.}

For $k\geq 3$, one possible approach to Problem 5 of \cite{FrProblems}  would be to prove that the set $S$ we construct in the proof of Theorem \ref{thm:Main} is actually a set of $k$-recurrence, or to prove that our construction necessarily results in a set which is not a set of $k$-recurrence. While our construction does not appear to restrict $\mu(A\cap T^{-n}A\cap T^{-2n}A\cap T^{-3n}A)$ for $n\in S$,  computations and estimates of
\begin{equation}\label{eqn:4term}
\lim_{N\to\infty} \frac{1}{N}\sum_{n=1}^N g(n^2\beta) \int f\cdot f\circ T^n \cdot f\circ T^{2n} \cdot f\circ T^{3n} \, d\mu
\end{equation}
analogous to those in \S\ref{sec:AffineLimits}-\S\ref{sec:MainProof} seem to require more intricate reasoning.  It may not be possible to specialize the limit in (\ref{eqn:4term}) to affine systems.  Perhaps one must consider arbitrary $2$-step totally ergodic nilsystems, or even more general systems.

For $k\geq 3$, our approach to Theorem 1.1 leads to the following natural conjecture, an analogue of Lemma \ref{lem:SqrtBHisrecurrent}.  Here $BH^{1/k}$ denotes $\{n\in \mathbb N: n^k\in BH\}$.
\begin{conjecture}\label{conj:NextStep}
  For all $\delta>0$, there exists $k_0\in \mathbb N$ such that for every $r\in \mathbb N$, every proper Bohr-Hamming Ball $BH:=BH(\bm\beta,\bm y,k,\varepsilon)$ with $k\geq k_0$, $\varepsilon>0$ and $y\in \mathbb T^r$, $BH^{1/k}$ is $(\delta,k)$-recurrent.
\end{conjecture}
Conjecture \ref{conj:NextStep} could be proved with appropriate higher-order analogues of Lemma \ref{lem:ReductionToWeyl}, Proposition \ref{prop:WeakLimit}, and Lemma \ref{lem:SmallFourierToW}.  For $k\geq 3$, it seems very unlikely that a reduction to $2$-step affine systems will be possible, and for $k\geq 4$, it is nearly certain that explicit computations must be carried out for essentially arbitrary $(k-1)$-step totally ergodic nilsystems.  These computations seem forbidding, so we hope a more qualitative approach can be developed.

\frenchspacing
\bibliographystyle{amsplain}
\bibliography{SquaresRecurrenceBib}

\providecommand{\bysame}{\leavevmode\hbox to3em{\hrulefill}\thinspace}
\providecommand{\MR}{\relax\ifhmode\unskip\space\fi MR }
\providecommand{\MRhref}[2]{%
  \href{http://www.ams.org/mathscinet-getitem?mr=#1}{#2}
}
\providecommand{\href}[2]{#2}
\begin{thebibliography}{10}

\bibitem{ABB}
Ethan Ackelsberg, Vitaly Bergelson, and Andrew Best, \emph{Multiple recurrence
  and large intersections for abelian group actions}, Discrete Anal. (2021),
  Paper No. 18, 91. \MR{4328788}

\bibitem{BHK}
Vitaly Bergelson, Bernard Host, and Bryna Kra, \emph{Multiple recurrence and
  nilsequences}, Invent. Math. \textbf{160} (2005), no.~2, 261--303, With an
  appendix by Imre Ruzsa. \MR{2138068}

\bibitem{BHMP}
Vitaly Bergelson, Bernard Host, Randall McCutcheon, and Fran\c{c}ois Parreau,
  \emph{Aspects of uniformity in recurrence}, vol. 84/85, 2000, Dedicated to
  the memory of Anzelm Iwanik, pp.~549--576. \MR{1784213}

\bibitem{EinsiedlerWard}
Manfred Einsiedler and Thomas Ward, \emph{Ergodic theory with a view towards
  number theory}, Graduate Texts in Mathematics, vol. 259, Springer-Verlag
  London, Ltd., London, 2011. \MR{2723325 (2012d:37016)}

\bibitem{ForrestThesis}
Alan~Hunter Forrest, \emph{Recurrence in dynamical systems: {A} combinatorial
  approach}, ProQuest LLC, Ann Arbor, MI, 1990, Thesis (Ph.D.)--The Ohio State
  University. \MR{2685439}

\bibitem{FrantzikinakisThreePoly}
Nikos Frantzikinakis, \emph{Multiple ergodic averages for three polynomials and
  applications}, Trans. Amer. Math. Soc. \textbf{360} (2008), no.~10,
  5435--5475. \MR{2415080}

\bibitem{FrProblems}
\bysame, \emph{Some open problems on multiple ergodic averages}, Bull. Hellenic
  Math. Soc. \textbf{60} (2016), 41--90. \MR{3613710}

\bibitem{FrKrPolyAffine}
Nikos Frantzikinakis and Bryna Kra, \emph{Polynomial averages converge to the
  product of integrals}, vol. 148, 2005, Probability in mathematics,
  pp.~267--276. \MR{2191231}

\bibitem{FLW2006}
Nikos Frantzikinakis, Emmanuel Lesigne, and M\'{a}t\'{e} Wierdl, \emph{Sets of
  {$k$}-recurrence but not {$(k+1)$}-recurrence}, Ann. Inst. Fourier (Grenoble)
  \textbf{56} (2006), no.~4, 839--849. \MR{2266880}

\bibitem{FLW2009}
\bysame, \emph{Powers of sequences and recurrence}, Proc. Lond. Math. Soc. (3)
  \textbf{98} (2009), no.~2, 504--530. \MR{2481957}

\bibitem{FurstenbergBook}
H.~Furstenberg, \emph{Recurrence in ergodic theory and combinatorial number
  theory}, Princeton University Press, Princeton, N.J., 1981, M. B. Porter
  Lectures. \MR{603625 (82j:28010)}

\bibitem{FK79}
H.~Furstenberg and Y.~Katznelson, \emph{An ergodic {S}zemer\'{e}di theorem for
  commuting transformations}, J. Analyse Math. \textbf{34} (1978), 275--291
  (1979). \MR{531279}

\bibitem{F77}
Harry Furstenberg, \emph{Ergodic behavior of diagonal measures and a theorem of
  {S}zemer\'{e}di on arithmetic progressions}, J. Analyse Math. \textbf{31}
  (1977), 204--256. \MR{498471}

\bibitem{GowersSz}
W.~T. Gowers, \emph{A new proof of {S}zemer\'{e}di's theorem}, Geom. Funct.
  Anal. \textbf{11} (2001), no.~3, 465--588. \MR{1844079}

\bibitem{griesmer2020separating}
John~T. Griesmer, \emph{Separating {B}ohr denseness from measurable
  recurrence}, Discrete Anal. (2021), Paper No. 9, 20. \MR{4308022}

\bibitem{GriesmerKrizInDiff}
\bysame, \emph{{Separating topological recurrence from measurable recurrence:
  exposition and extension of Kriz's example}}, arXiv e-prints (2021),
  arXiv:2108.01642.

\bibitem{HostKraBook}
Bernard Host and Bryna Kra, \emph{Nilpotent structures in ergodic theory},
  Mathematical Surveys and Monographs, vol. 236, American Mathematical Society,
  Providence, RI, 2018. \MR{3839640}

\bibitem{KuipersNiederreiter}
L.~Kuipers and H.~Niederreiter, \emph{Uniform distribution of sequences},
  Wiley-Interscience [John Wiley \& Sons], New York-London-Sydney, 1974, Pure
  and Applied Mathematics. \MR{0419394}

\bibitem{Kriz}
Igor K\v{r}\'{\i}\v{z}, \emph{Large independent sets in shift-invariant graphs:
  solution of \text{Bergelson's} problem}, Graphs Combin. \textbf{3} (1987),
  no.~2, 145--158. \MR{932131}

\bibitem{McCutcheonAlexandria}
Randall McCutcheon, \emph{Three results in recurrence}, Ergodic theory and its
  connections with harmonic analysis ({A}lexandria, 1993), London Math. Soc.
  Lecture Note Ser., vol. 205, Cambridge Univ. Press, Cambridge, 1995,
  pp.~349--358. \MR{1325710}

\bibitem{McCutcheonBook}
\bysame, \emph{Elemental methods in ergodic {R}amsey theory}, Lecture Notes in
  Mathematics, vol. 1722, Springer-Verlag, Berlin, 1999. \MR{1738544}

\bibitem{RothFrench}
K.~F. Roth, \emph{Sur quelques ensembles d'entiers}, C. R. Acad. Sci. Paris
  \textbf{234} (1952), 388--390. \MR{46374}

\bibitem{RothEnglish}
\bysame, \emph{On certain sets of integers}, J. London Math. Soc. \textbf{28}
  (1953), 104--109. \MR{51853}

\bibitem{RudinGroupsBook}
Walter Rudin, \emph{Fourier analysis on groups}, Interscience Tracts in Pure
  and Applied Mathematics, No. 12, Interscience Publishers (a division of John
  Wiley and Sons), New York-London, 1962. \MR{0152834}

\bibitem{WeissBook}
Benjamin Weiss, \emph{Single orbit dynamics}, CBMS Regional Conference Series
  in Mathematics, vol.~95, American Mathematical Society, Providence, RI, 2000.
  \MR{1727510}

\bibitem{Weyl1916}
Hermann Weyl, \emph{\"{U}ber die {G}leichverteilung von {Z}ahlen mod. {E}ins},
  Math. Ann. \textbf{77} (1916), no.~3, 313--352. \MR{1511862}

\end{thebibliography}

\end{document}